\documentclass{amsart}
\usepackage{amssymb}
\usepackage{hyperref}
\usepackage{bbm}
\newcommand{\R}{\mathbb{R}}
\newcommand{\C}{\mathbb{C}}
\newcommand{\Z}{\mathbb{Z}}

\DeclareMathOperator{\SPAN}{span}

\newtheorem{theorem}{Theorem}

\newtheorem{lemma}{Lemma}

\theoremstyle{definition}
\newtheorem{remark}{Remark}

\title[Non-existence of internal mode for the Zakharov system]
{Non-existence of internal mode for small solitary waves of the 1D Zakharov system}
\author{Yvan Martel}
\address{Laboratoire de Mathématiques de Versailles, UVSQ, Université Paris-Saclay, CNRS,
78000 Versailles, France and Institut Universitaire de France}
\email{yvan.martel@uvsq.fr}

\author{Guillaume Rialland}
\address{Laboratoire de Mathématiques de Versailles, UVSQ, Université Paris-Saclay, CNRS,
78000 Versailles, France}
\email{guillaume.rialland@uvsq.fr}

\begin{document}

\begin{abstract}
We prove that the linearised operator around any sufficiently small solitary wave of the one-dimensional Zakharov system has no internal mode. This spectral result, along with its proof, is expected to play a role in the study of the asymptotic stability of solitary waves.
\end{abstract}

\maketitle

\section{Introduction and main result}

In this article, we consider the one-dimensional scalar Zakharov system, which we write 
under the following form
\begin{equation}\label{ZL}
\left\{ \begin{aligned}i\partial_t u &= - \partial_x^2 u - nu \\
\partial_t n &= - \partial_x v \\
\partial_t v &= - \partial_x n + \partial_x (|u|^2 )
\end{aligned}\right.
\end{equation}
where $u:(t , x) \in \R \times \R\mapsto \C$,
$n:(t , x) \in \R \times \R\mapsto \R$, $v:(t , x) \in \R \times \R\mapsto \R$.
This equation was introduced by V.E. Zakharov in~\cite{Za} to describe the propagation of Langmuir turbulence in a plasma. We also refer to \cite{Gib,SuSu} for the derivation of this equation.

We observe that for a solution $(u , v , n)$ to \eqref{ZL}, three quantities,
respectively called the mass, the energy and the momentum, are formally preserved through time: 
\[
\int_{\R} |u|^2,\quad
\int_{\R} \left( | \partial_x u |^2 - n |u|^2 + \frac{n^2}{2} + \frac{v^2}{2} \right),\quad
\Im \left( \int_{\R} \overline{u} \partial_x u \right) + \int_{\R} nv.
\]
The Cauchy problem associated to \eqref{ZL} is globally well-posed in the energy space, \emph{i.e.}
for any initial data $(u_0, n_0 ,v_0) \in H^1( \R ; \C ) \times L^2( \R ; \R ) \times L^2( \R ; \R )$; see \cite{Bou,Gin,Sa}. 
We also recall the phase and translation invariances for the system \eqref{ZL}: if $(u , n , v )$ is a solution of~\eqref{ZL}, then, for any $\sigma$,$\gamma \in \R$, $(t,x)\mapsto(u(t , x- \sigma ) e^{i \gamma} , n(t , x- \sigma ) , v(t , x- \sigma ))$ is also a solution of~\eqref{ZL}.

As discussed in \cite[Eq. (1.10)]{Kiv}, for small solutions, the scalar Zakharov system can be seen as a perturbation of the one-dimensional cubic Schrödinger equation
\begin{equation}\label{nls}
i \partial_t u + \partial_x^2 u + |u|^2u =0.
\end{equation}
Indeed, if $(u,n,v)$ is a solution of \eqref{ZL}, then for any $\omega>0$, setting
\begin{equation}\label{eq:sg}
u(t,x)=\sqrt{\omega}\widetilde u(\omega t,\sqrt{\omega}x),\
n(t,x)={\omega}\widetilde n(\omega t,\sqrt{\omega}x),\
v(t,x)={\omega}\widetilde v(\omega t,\sqrt{\omega}x),
\end{equation}
the triple $(\widetilde u,\widetilde n,\widetilde v)$ satisfies
\begin{equation}\label{eq:zl}
\left\{\begin{aligned} i \partial_t \widetilde u 
 & = - \partial_x^2 \widetilde u - \widetilde n \widetilde u \\ 
 \sqrt{\omega}\partial_t \widetilde n &= -\partial_x \widetilde v\\
 \sqrt{\omega}\partial_t \widetilde v &= - \partial_x \widetilde n + \partial_x ( |\widetilde u|^2 ) 
 \end{aligned}\right.
\end{equation}
For $\omega>0$ small, the third line of the system formally implies that
$\widetilde n\approx |\widetilde u|^2$ which, inserted in the first line of the system,
says that $\widetilde u$ approximately satisfies the Schrödinger equation~\eqref{nls}.

The Zakharov system \eqref{ZL} admits standing solitary waves
(see for instance~\cite{Kiv,Oh,Wu}). For any $\omega>0$, set
\begin{equation}\label{eq:QQ}
\phi_\omega (x) = \sqrt{\omega} Q ( \sqrt{\omega} x )
\quad \mbox{where}\quad Q (y) = \frac{\sqrt{2}}{\text{cosh} (y)}
\quad \mbox{satisfies}\quad Q''+Q^3 = Q.
\end{equation}
Then, $(u , n , v)(t,x) = (e^{i \omega t} \phi_\omega(x) ,\phi_\omega^2(x),0)$ is a solution of~\eqref{ZL}.
Moreover, although the one-dimensional Zakharov system is not invariant
by any Galilean or Lorentz-type transformation, 
it admits a family of travelling waves, explicitly given by
\begin{equation}\label{eq:sl}
\left\{ \begin{aligned} u(t , x) &= \sqrt{1-\beta^2}\, \phi_\omega (x-\beta t)
\exp(i \Gamma(t,x))\\ 
n (t , x) &= \phi_\omega^2 (x-\beta t) \\
v (t , x) &= \beta \phi_\omega^2 (x-\beta t) \end{aligned} \right.
\end{equation}
for any $\omega > 0$, $\beta \in (-1,1)$, where
\[
\Gamma(t,x)= \omega t- \frac{\beta^2}4 t + \frac\beta2 x .
\]
As for the cubic Schrödinger equation \eqref{nls}, all the solitary waves of~\eqref{ZL} defined above are known to be orbitally stable in the energy space; see \cite{Oh,Wu}.

Now, we turn to the question of asymptotic stability of solitary waves.
In the general context of nonlinear Schrödinger equations,
we refer to \cite{BuPe,PKA} for pioneering works on the subject 
and to the reviews \cite{Cu2,Ge,Ma4,Ma5,Schlag}, for instance.
Here, we focus on certain one-dimensional models that are close to \eqref{nls}.
First, recall that solitary waves of the cubic equation \eqref{nls} are not asymptotically stable in the energy space (see the Introduction of \cite{Ma1}). However, asymptotic stability was proved in a refined topology
of weighted spaces, using tools from the integrability theory~\cite{CuPe} or using 
more general techniques involving the distorted Fourier transform~\cite{LiLu}.

Second, recent articles have addressed the question of asymptotic stability of solitary waves
for semilinear perturbations of \eqref{nls}, showing that such perturbations 
could significantly change the situation in the energy space.
For example, the small solitary waves of the model
\begin{equation}
i\partial_t u + \partial_x^2 u + |u|^2 u + g(|u|^2)u =0
\label{NLSg}
\end{equation}
are known to be locally asymptotically stable for a wide class of perturbations $g \not\equiv 0$,
which satisfy $g(s)=o(s)$ and $g\leqslant 0$ in a certain sense, in particular for $g(s) = - s^{q-1}$ for any $q>2$. We refer to \cite{Ma1} for the special case $g(s)=-s^2$ and to~\cite{GR1} for the general case.
Note that for such models, it is proved that there is no internal mode, i.e. there exists no non-trivial time-periodic solution of the linearised problem around the solitary wave.
The article \cite{CoGe} deals with more general situations where it is assumed that there
is no internal mode, and with a stronger notion of asymptotic stability
(full asymptotic stability \emph{versus} local asymptotic stability,
see \cite{Ge} for a discussion).

Lastly, for the model \eqref{NLSg}, in the case where $g \not\equiv 0$,
$g\geqslant 0$ in a certain sense, satisfies $g(s)=o(s)$, in particular, for $g(s) = s^{q-1}$ for any $q>2$,
the asymptotic stability of solitary waves was proved in \cite{Ma2} (for $g(s)=s^2$) and \cite{GR2}.
 In that case, we emphasize that the presence of an internal mode,
defined as a time-periodic solution of linearised equation around the solitary wave, makes the analysis considerably
more involved (see \cite{PKA} for a pioneering insight on such questions).
The situation is similar for the model
$i\partial_t u + \partial_x^2 u+ |u|^{p-1} u=0$, 
for $p\neq 3$ close to $3$,
for which an internal mode is also present, see \cite{CoGu,Cu1}.
Therefore, one can say that the case of semilinear perturbations of the integrable equation \eqref{nls} is now
rather well-understood.

As observed above, the Zakharov system \eqref{ZL} is also a perturbation of~\eqref{nls} for small solutions, even though 
of different nature.
It is thus natural to study the asymptotic stability of its solitary waves, 
starting with the potential issue of existence of internal modes.
Actually, the purpose of this paper is to prove that there exists no internal mode
for small solitary waves of \eqref{ZL}.
To state a precise result,
we linearise the system~\eqref{ZL} around a solitary wave of the form \eqref{eq:sl}, also changing space and time variables
to make the function $Q$ appear and to highlight the small parameter $\omega$.

For $\omega>0$ and $\beta\in (-1,1)$, we decompose a solution $(u,n,v)$ of \eqref{ZL}
around the travelling solitary wave defined in \eqref{eq:sl}, by setting
\begin{equation*}
\left\{
\begin{aligned}
 u(t , x) 
&=\sqrt{\omega} \sqrt{1-\beta^2}  (Q + U ) ( \omega t , \sqrt{\omega}( x-\beta t)) 
\exp(i\Gamma(t,x))\\
 n(t , x) 
&= \omega ( Q^2 + 2QU_1 + N) ( \omega t , \sqrt{\omega}( x-\beta t)) \\
 v(t , x) 
&= \omega ( \beta Q^2+2\beta QU_1 + V ) (\omega t , \sqrt{\omega}( x-\beta t))
\end{aligned}\right.
\end{equation*}
for unknown small functions $U(s , y)=U_1(s,y)+iU_2(s,y)$, $N(s , y)$ and $V(s , y)$.
Note that the presence of the term $2QU_1$ in the decompositions of $n$ and $v$ is natural
since $|Q+U|^2 = Q^2+ 2 QU_1+|U|^2$.
Define the operators $L_+$ and $L_-$ by
\begin{equation}\label{eq:LL}
 L_+=-\partial_y^2+1-3Q^2,\quad L_-= -\partial_y^2+1-Q^2,
\end{equation}
and recall the well-known property (\cite{We})
\[
\ker L_+=\SPAN (Q'),\quad 
\ker L_-=\SPAN (Q).
\]
Discarding nonlinear terms in $(U,N,V)$, we find a linear system for $(U_1,U_2,N,V)$
\begin{equation}
\left\{ \begin{aligned} \partial_s U_1 &= L_- U_2 \\ 
\partial_s U_2 &= - L_+ U_1 +QN \\
\sqrt{\omega} ( 2 Q \partial_s U_1 +\partial_s N ) &= \beta \partial_y N - \partial_y V \\ 
\sqrt{\omega} (2\beta Q \partial_s U_1 + \partial_s V )&= - \partial_y N + \beta \partial_y V
\end{aligned} \right.
\label{LinZ}
\end{equation}
By definition, an internal mode is a time-periodic solution $(U_1,U_2,N,V)$ of the linear system \eqref{LinZ}. The main result of the present article is the following. 

\begin{theorem}\label{thmPer}
Let $\beta \in (-1 \, , 1)$. If $\omega >0$ is sufficiently small then $(U_1 , U_2 , N , V) \in C^0 ( \R , H^1 ( \R )^2 \times L^2 ( \R )^2 )$ is a time-periodic solution of the system \eqref{LinZ} if and only if there exist $a_1,a_2 \in \R$ such that $U_1(s,y) = a_1 Q'(y)$, $U_2(s,y)=a_2Q(y)$ and $N=V=0$.
\end{theorem}

Using Fourier decomposition for time-periodic functions, we actually only need to investigate the unimodal case
\begin{equation*}
\left\{
\begin{aligned}
 U_1 (s , y) &= \cos ( \lambda s ) C_1 (y) + \sin ( \lambda s ) S_1 (y) \\
 U_2 (s , y) &= \cos ( \lambda s ) C_2 (y) + \sin ( \lambda s ) S_2 (y) \\
 N (s , y) &= \cos ( \lambda s ) C_N (y) + \sin ( \lambda s ) S_N (y) \\
V(s , y ) &= \cos ( \lambda s ) C_V (y) + \sin ( \lambda s ) S_V (y)
\end{aligned}\right.
\end{equation*}
(For $\lambda=0$, the functions $S_1$, $S_2$, $S_N$ and $S_V$ are useless and taken to $0$.)
Inserting this form into~\eqref{LinZ}, we find that the functions
 $C_1,$ $S_1$, $C_2$, $S_2$ and $C_N$, $S_N$, $C_V$, $S_V$ must satisfy
\begin{equation}\label{IM1}
\left\{ \begin{aligned} 
L_- C_2 &= \lambda S_1\\ 
L_+ S_1 &= \lambda C_2 + QS_N \\ 
L_- S_2 & = - \lambda C_1 \\ 
L_+ C_1 &= - \lambda S_2 + QC_N \end{aligned} \right.
\end{equation}
and
\begin{equation}\label{IM2}
\left\{ \begin{aligned} 
& C_V' - \beta C_N' = -\lambda \sqrt{\omega}(S_N + 2 Q S_1 ) \\
& S_V' - \beta S_N' = \lambda \sqrt{\omega}(C_N + 2 Q C_1 ) \\
& C_N' - \beta C_V' = - \lambda \sqrt{\omega}( S_V +2 \beta Q S_1 ) \\
& S_N' - \beta S_V' = \lambda \sqrt{\omega}(C_V + 2 \beta Q C_1 )
\end{aligned} \right.
\end{equation}
In the formulation \eqref{IM1}-\eqref{IM2}, 
the function $Q$, defined in \eqref{eq:QQ} and the operators $L_\pm$, defined in \eqref{eq:LL}, are fixed,
and thus one easily sees the influence of the various parameters:
the eigenvalue $\lambda$, the speed parameter $\beta\in (-1,1)$,
and the small parameter $\omega>0$, related to the size of the solitary wave in the original variables $(t,x)$.

Theorem \ref{thmPer} is a consequence of the fact that the only non-zero solutions of 
the system~\eqref{IM1}-\eqref{IM2}
are the trivial ones given by the respective kernel of the operators~$L_+$ and~$L_-$.

\begin{theorem}\label{thmIM}
Let $\beta \in (-1,1)$.
If $\omega >0$ is sufficiently small then the only solutions $(\lambda , C_1 , S_1 , C_2 , S_2 , C_N , S_N , C_V , S_V ) \in \R \times H^1 ( \R )^4 \times L^2 ( \R )^4$ of the system~\eqref{IM1}-\eqref{IM2} are
\begin{equation*}
\lambda = 0, \quad
C_1 \in \SPAN (Q') , \quad C_2 \in \SPAN (Q) , \quad
C_N = C_V = 0
\end{equation*}
and
\begin{equation*}
\lambda \in \R , \quad
C_1=S_1=C_2=S_2=C_N=S_N=C_V=S_V=0.
\end{equation*}
\end{theorem}

\begin{remark}\label{rk:1bis}
In the special case $\beta=0$,
the system \eqref{IM1}-\eqref{IM2} splits into two identical decoupled systems satisfied by 
$(\lambda,S_1,C_2,S_N)$ and $(-\lambda,C_1,S_2,C_N)$, of the form
\begin{equation*}
 \left\{\begin{aligned}
 L_-C_2 & = \lambda S_1\\
 L_+S_1 & = \lambda C_2+QS_N\\
 S_N'' + \lambda^2 \omega S_N& = - 2 \lambda^2 \omega QS_1
 \end{aligned}\right.
\end{equation*}
Restricted to $\beta=0$, the proof of Theorem~\ref{thmIM} would be algebraically 
simpler, but it would follow the same steps.
\end{remark}
\begin{remark}\label{rk:2}
From the proof of Theorem~\ref{thmIM}, it follows that 
for any $\beta\in (-1,1)$ and for $\omega>0$ sufficiently small,
there exists no nontrivial solution of \eqref{IM1}-\eqref{IM2} with $\lambda=1$ and such that
\[
C_1,S_1,C_2 , S_2 , C_N , S_N , C_V , S_V\in L^\infty,\quad
C_1',S_1',C_2' , S_2' , C_N' , S_N' , C_V' , S_V'\in L^2.
\]
This means that there exists no resonance at the edge of the
continuous spectrum. See \S\ref{prk} for a justification.
Note that this is in contrast with
the cubic Schrödinger equation \eqref{nls}, for which a resonance is known
(see Remark~\ref{rk:3} below). 
Therefore, 
the Zakharov system, seen as a perturbation of the cubic 
Schrödinger equation~\eqref{nls}, for small solitary waves, 
makes the resonance disappear and no internal mode emerge.
This favorable spectral property regarding the
asymptotic stability is thus similar to that of equation \eqref{NLSg} 
for general non-zero negative perturbations $g$ treated in \cite{Ma1,GR1},
and should play a role in any attempt to address the question of the
asymptotic stability of the small solitary waves of \eqref{ZL}.

We refer to \cite{Mar} for related asymptotic results for the evolution system (1) under a symmetry assumption on the solution.
In summary, the results  in \cite{Mar} show the local asymptotic stability of the zero solution in various space-time regions.
\end{remark}

Before proving Theorem \ref{thmIM} in the rest of this article,
we check that it formally implies Theorem \ref{thmPer}.
We refer to the Appendix for a complete proof of this fact.

Let $(U_1,U_2,N,V) \in C^0 (\R , H^1 ( \R )^2 \times L^2 ( \R )^2)$ be a $T$-periodic solution of system~\eqref{LinZ}. We use a Fourier decomposition in the time variable
\begin{align*}
U_1 (s,y) &= \sum\limits_{n=0}^{+ \infty} \left ( \cos ( \lambda_n s ) C_1^{(n)} (y) + \sin ( \lambda_n s ) S_1^{(n)} (y) \right ) \\
U_2 (s,y) &= \sum\limits_{n=0}^{+ \infty} \left ( \cos ( \lambda_n s ) C_2^{(n)} (y) + \sin ( \lambda_n s ) S_2^{(n)} (y) \right ) \\
N(s,y) &= \sum\limits_{n=0}^{+ \infty} \left ( \cos ( \lambda_n s ) C_N^{(n)} (y) + \sin ( \lambda_n s ) S_N^{(n)} (y) \right ) \\
V(s,y) &= \sum\limits_{n=0}^{+ \infty} \left ( \cos ( \lambda_n s ) C_V^{(n)} (y) + \sin ( \lambda_n s ) S_V^{(n)} (y) \right )
\end{align*}
where $\lambda_n = \frac{2 \pi n}{T}$ and $S_1^{(0)}=S_2^{(0)}=S_N^{(0)}=S_V^{(0)}=0$.

Inserting formally 
this expansion into the system \eqref{LinZ},
we find that, for all $n \geqslant 0$, the tuple $(\lambda_n , C_1^{(n)} , S_1^{(n)} , C_2^{(n)} , S_2^{(n)} , C_N^{(n)} , S_N^{(n)} , C_V^{(n)} , S_V^{(n)})$ satisfies the system~\eqref{IM1}-\eqref{IM2}. 
For $n \neq 0$, it follows from Theorem \ref{thmIM} that $C_1^{(n)} = S_1^{(n)} = C_2^{(n)} = S_2^{(n)} = C_N^{(n)} = S_N^{(n)} = C_V^{(n)} = S_V^{(n)} = 0$.
For $n=0$, it follows from Theorem \ref{thmIM} that $C_1^{(0)}  \in \text{span} (Q')$, $C_2^{(0)}  \in \text{span} (Q)$ and $C_N^{(0)}  = C_V^{(0)}  = 0$. Hence, $U_1 (s) = C_1^{(0)} \in \text{span} (Q')$, $U_2 (s) = C_2^{(0)}  \in \text{span} (Q)$ and $N(s,y)=V(s,y)=0$.

\subsection*{Notation} We denote $\langle f , g \rangle = \text{Re} \int_{\R} f \overline{g}$ and we use the notation $\| \cdot \|$ for the $L^2$-norm. The letter $C$ will denote various positive constants, independent of $s$, $y$, $\omega$, $\beta$ and $\lambda$, whose expression may change from one line to another; if needed, $C'$ and $C''$ will denote additional constants.
We will also use the notation $A\lesssim B$ when the inequality
$A\leqslant CB$ holds for such a constant $C$.
 
\section{Basic spectral properties}
 We recall from \cite{We} the following positivity properties,
for any $f\in H^1(\R)$,
\begin{equation}\label{4.spec}
\begin{aligned}
\langle L_+ f, f \rangle & \geqslant
C \| f\|_{H^1}^2 - C' \left( \langle f , Q \rangle^2 + \langle f , yQ \rangle^2 \right) \\
\langle L_- f , f \rangle & \geqslant
C \|f\|_{H^1}^2 - C' \langle f , \Lambda Q \rangle^2 
\end{aligned}
\end{equation}
where we have defined the function $\Lambda Q = \frac{1}{2} (Q+yQ')$.

Define the following operators
\[
S = \partial_y - \frac{Q'}{Q},\quad
S^* = - \partial_y - \frac{Q'}{Q},\quad 
M = - \partial_y^2 + 1.
\]
It is standard to observe that $L_-=S^*S$.
We also recall a factorisation property from \cite[Lemma 2]{Ma1} (see also \cite{CGNT})
\begin{equation}
S^2 L_+ L_- = M^2 S^2.
\label{facto}
\end{equation}
This factorisation will enable us to pass from a problem formulated in terms
of $L_\pm$ to a transformed problem involving the operator $M$ only.
Being without potential and having a trivial kernel, the operator $M$ is simpler to analyse by
virial arguments.
\begin{remark}\label{rk:3}
We look for the NLS limit $\omega\downarrow 0$ in the system \eqref{IM1}-\eqref{IM2}, \emph{i.e.}
\begin{equation*}
\left\{ \begin{aligned} 
L_- C_2 & = \lambda S_1 \\ 
L_+ S_1 & = \lambda C_2 + QS_N \\ 
L_- S_2 & = -\lambda C_1 \\ 
L_+ C_1 & = -\lambda S_2 + QC_N
\end{aligned} \right. \qquad 
\left\{ \begin{aligned} 
& (C_V -\beta C_N)' = 0 \\
& (S_V -\beta S_N)' = 0 \\
& (C_N -\beta C_V)' = 0 \\
& (S_N -\beta S_V)' = 0
\end{aligned} \right.
\end{equation*}
This leads to $C_N=C_V=S_N=S_V=0$ and to the two independent systems
\begin{equation*}
\left\{ \begin{aligned} 
L_+ S_1 &= \lambda C_2 \\
L_- C_2 &= \lambda S_1 
\end{aligned} \right.\qquad
\left\{ \begin{aligned} 
L_+ C_1 &= - \lambda S_2 \\
L_- S_2 & = - \lambda C_1 
\end{aligned} \right.\label{IM9}
\end{equation*}
The only non-trivial solutions $(\lambda,C_1,S_1,C_2,S_2)
\in \R\times H^1(\R)^4$
are $\lambda=0$, $C_1,S_1 \in \SPAN (Q')$, $C_2,S_2 \in \SPAN (Q)$.
However, there exists a resonance for $\lambda=1$ (see \cite{CoGu}),
\[
S_1= \mu_1(1-Q^2) ,\quad C_2= \mu_1,\quad S_2= \mu_2(1-Q^2) ,\quad C_1=\mu_2.
\]
Indeed, note that the first system with $\lambda=1$ gives $L_+L_- C_2= C_2$.
By \eqref{facto} and setting $W_2=S^2C_2$, this yields 
$M^2 W_2= W_2$ and thus $W_2=1$ (up to a multiplicative constant).
Then, $S^2 1=1$ says that $C_2=1$ (up to a multiplicative constant
and up to the explicit kernel).
Lastly, using the system again, we have $S_1=L_- 1 = 1-Q^2$.
\end{remark}

For future use, we define an auxiliary function $h$.
\begin{lemma} \label{lemh}
Define the function $h:\R\to(0,+\infty)$ by
\[
h(y) = \frac{1}{Q(y)} \int_y^{+ \infty} zQ^2(z)\, \textnormal{d}z.
\]
It holds 
\begin{itemize}
\item For all $y \in \R$, $0<h(y) \leqslant C (1+|y|)Q(y)$.
\item $(S^*)^2 h = -(Q+2yQ')$
\item For all $w \in H^1 (\R)$, 
\[ \int_{\R} Q^{\frac{1}{2}} w^2 \lesssim \left( \int_{\R} hw \right)^2 + \int_{\R} (w')^2. \]
\end{itemize}
\end{lemma}

\begin{proof} 
It will not be used, but the function $h$ has the following explicit expression
$h = \frac{1}{Q} \left( 3 \ln 2 -2 \ln Q + 2yQ'/{Q} \right)$.

First, 
for $y \geqslant 0$, $h(y) \lesssim e^y \int_y^{+ \infty} ze^{-2z} \, \text{d}z \lesssim (1+y) e^{-y} 
\lesssim (1+y) Q(y)$. Moreover, $h$ is even. For the second point, we have
\[ (S^*)^2 h = \frac{1}{Q} (Qh)'' = \frac{1}{Q} (-yQ^2)' = -(Q+2yQ'). \]
Now let us prove the third point of the lemma. Take $w \in H^1 ( \R )$ and begin with $w(y) = w(z) + \int_z^y w'$. Multiplying by $h(z)$ and integrating in $z\in\R$, it follows that
\[
w(y) \int_{\R} h = \int_{\R} hw + \int_{\R} h(z) \left( \int_z^y w' \right) \text{d}z.
\]
Thus,
\begin{align*} 
w^2(y) & \lesssim \left( \int_{\R} hw \right)^2 + \left( \int_{\R} h(z) \left( \int_z^y w' \right) \, \text{d}z \right)^2 \\
& \lesssim \left( \int_{\R} hw \right)^2 + \left( \int_{\R} |h(z)| \left( |y|^{\frac{1}{2}} + |z|^{\frac{1}{2}} \right) \text{d}z \right)^2 \int_{\R} (w')^2 \\
 & \lesssim \left( \int_{\R} hw \right)^2 + (1+|y|) \int_{\R} (w')^2
\end{align*}
Multiplying the above inequality by $Q^{\frac{1}{2}} (y)$ and integrating in $y\in \R$,
we obtain the inequality.
Note that the property proved above is not specific to the choice of the function $h$
and holds for any function with sufficient decay and a non zero integral.
\end{proof}

\section{Proof of non-existence of internal mode}

We observe that for any solution of \eqref{IM1}-\eqref{IM2} in the sense of distributions, 
assuming for example that $ C_1 , S_1 , C_2 , S_2 , C_N , S_N , C_V , S_V\in L^2(\R)$,
and using the system of equations, we obtain that $ C_1 , S_1 , C_2 , S_2 , C_N , S_N , C_V , S_V\in H^s(\R)$
for any $s\geqslant 0$.
We consider such a non trivial solution of  \eqref{IM1}-\eqref{IM2}.

\subsection{Almost orthogonality and resolution of a subsystem}\label{S:3.2}

We show here that the subsystem \eqref{IM2} provides pseudo-orthogonality relations that will be helpful in order to analyse the subsystem \eqref{IM1}.
First, we observe that the subsystem \eqref{IM2} is equivalent to
\begin{equation}\label{IM3}
\begin{pmatrix} C_V' \\ S_V' \\ C_N' \\ S_N' \end{pmatrix} 
= \varepsilon \mathbf{A} \begin{pmatrix} C_V \\ S_V \\ C_N \\ S_N \end{pmatrix}
+\varepsilon \kappa \begin{pmatrix} -QS_1 \\ QC_1 \\ -\gamma QS_1 \\ \gamma QC_1 \end{pmatrix}
\end{equation}
where 
\[
\varepsilon = \lambda \sqrt{\omega},\quad 
\gamma = \frac{2\beta}{1+\beta^2}\in (-1,1), \quad \kappa=\frac{2 (1+\beta^2 )}
{1- \beta^2}
\]
and
\[
\mathbf{A} = \frac{1}{1-\beta^2} \begin{pmatrix} 0 & - \beta & 0 & -1 \\ \beta & 0 & 1 & 0 \\ 0 & -1 & 0 & -\beta \\ 1 & 0 & \beta & 0 \end{pmatrix}. 
\]
The matrix $\mathbf{A} $ has four imaginary eigenvalues (counted with multiplicity): $ \pm\frac{i }{1 \pm \beta}$ and $\mathbf{A} = \mathbf{P} \mathbf{D}\mathbf{P}^{-1}$ where
\begin{align*}
& \mathbf{P} = \begin{pmatrix}i & -i & i & -i \\ -1 & -1 & 1 & 1 \\ -i & i & i & -i \\ 1 & 1 & 1 & 1 \end{pmatrix} , \quad \mathbf{D} = \begin{pmatrix} \frac{i }{1+\beta} & 0 & 0 & 0 \\ 0 & - \frac{i }{1 + \beta} & 0 & 0 \\ 0 & 0 & \frac{i }{1-\beta} & 0 \\ 0 & 0 & 0 & - \frac{i }{1- \beta} \end{pmatrix}.
\end{align*}
Set
\[
\begin{pmatrix} Y_1 \\ Y_2 \\ Y_3 \\ Y_4 \end{pmatrix} = \mathbf{P}^{-1} \begin{pmatrix} C_V \\ S_V \\ C_N \\ S_N \end{pmatrix}
\]
so that the system \eqref{IM3} and the diagonalisation of $\mathbf{A}$ lead to
\begin{equation}
\begin{split}
\begin{pmatrix} Y_1 ' \\ Y_2 ' \\ Y_3 ' \\ Y_4 ' \end{pmatrix} 
&= \varepsilon \mathbf{D} \begin{pmatrix} Y_1 \\ Y_2 \\ Y_3 \\ Y_4 \end{pmatrix}
+ \varepsilon\kappa \mathbf{P}^{-1} \begin{pmatrix} -QS_1 \\ QC_1 \\ -\gamma QS_1 \\\gamma QC_1 \end{pmatrix} \\
&= \varepsilon \begin{pmatrix} \frac{i }{1+\beta} Y_1 \\[4pt] -\frac{i }{1+\beta}Y_2 \\[4pt] \frac{i }{1-\beta} Y_3 \\[4pt] -\frac{i }{1-\beta} Y_4 \end{pmatrix} 
+ \frac{\varepsilon \kappa}{4} \begin{pmatrix} -(1- \gamma ) Q (C_1-iS_1) \\ -(1- \gamma ) Q(C_1+iS_1) \\ (1+ \gamma ) Q(C_1 + iS_1) \\ (1+ \gamma ) Q(C_1 -iS_1) \end{pmatrix}.
\end{split}
\label{IM4}
\end{equation}
We observe that
\[ 
 e^{- \frac{i \varepsilon y}{1+ \beta}} \left( Y_1 ' - \frac{i \varepsilon}{1+\beta} Y_1 \right)
= \frac{\text{d}}{\text{d} y} \left( e^{- \frac{i \varepsilon y}{1+\beta}} Y_1\right)
\]
and thus, by the first line of system \eqref{IM4}
and $Y_1\in H^1(\R)$, we obtain
\[
\int_{\R} e^{- \frac{i \varepsilon y}{1+\beta}} Q (C_1 -iS_1) \, \text{d}y =0.
\]
Taking the real and imaginary parts of the above identity yields
\begin{align*}
& \int_{\R} \cos \left( \frac{\varepsilon y}{1+\beta} \right) Q(y) C_1(y) \, \text{d}y - \int_{\R} \sin \left( \frac{\varepsilon y}{1+\beta} \right) Q(y) S_1 (y) \, \text{d}y = 0 \\
& \int_{\R} \cos \left( \frac{\varepsilon y}{1+\beta} \right) Q(y) S_1(y) \, \text{d}y + \int_{\R} \sin \left( \frac{\varepsilon y}{1+ \beta} \right) Q(y) C_1 (y) \, \text{d}y = 0.
\end{align*}
Using the second line of system \eqref{IM4} yields the same relations, while the third and fourth lines give two other relations. 
We gather below the four relations obtained
\begin{equation}
\begin{aligned}
& \int_{\R} \cos \left( \frac{\varepsilon y}{1+\beta} \right) QC_1 \, \text{d}y = \int_{\R} \sin \left( \frac{\varepsilon y}{1+\beta} \right) QS_1 \, \text{d}y \\
& \int_{\R} \sin \left( \frac{\varepsilon y}{1+ \beta} \right) QC_1 \, \text{d}y = - \int_{\R} \cos \left( \frac{\varepsilon y}{1+\beta} \right) QS_1 \, \text{d}y \\
& \int_{\R} \cos \left( \frac{\varepsilon y}{1-\beta} \right) QC_1 \, \text{d}y = - \int_{\R} \sin \left( \frac{\varepsilon y}{1- \beta} \right) QS_1 \, \text{d}y \\
 & \int_{\R} \sin \left( \frac{\varepsilon y}{1-\beta} \right) QC_1 \, \text{d}y = \int_{\R} \cos \left( \frac{\varepsilon y}{1-\beta} \right) QS_1 \, \text{d}y.
\end{aligned}
\label{IMorthos}
\end{equation}
Moreover, system \eqref{IM3} yields an explicit expression for $C_V$, $S_V$, $C_N$ and $S_N$ in terms of $S_1$ and $C_1$, which we establish now. In what follows, we use the following condensed notation: 
\begin{align*}
& s^{\pm} (y) = \frac{1}{2} \left( \sin \left( \frac{\varepsilon y}{1 + \beta} \right) \pm \sin \left( \frac{\varepsilon y}{1-\beta} \right) \right), \\
& c^{\pm} (y) = \frac{1}{2} \left( \cos \left( \frac{\varepsilon y}{1 + \beta} \right) \pm \cos \left( \frac{\varepsilon y}{1-\beta} \right) \right), \\
& s_\gamma^{\pm} (y) = \frac{1- \gamma}{2} \sin \left( \frac{\varepsilon y}{1+\beta} \right) \pm \frac{1 + \gamma}{2} \sin \left( \frac{\varepsilon y}{1-\beta} \right), \\
 & c_\gamma^{\pm} (y) = \frac{1- \gamma}{2} \cos \left( \frac{\varepsilon y}{1+\beta} \right) \pm \frac{1 + \gamma}{2} \cos \left( \frac{\varepsilon y}{1-\beta} \right).
\end{align*}
We compute
\[ \exp \left( \varepsilon y \mathbf{A} \right) = \mathbf{P} e^{ \varepsilon y \mathbf{D}} \mathbf{P}^{-1} = \begin{pmatrix} c^+ & s^- & -c^- & -s^+ \\ -s^- & c^+ & s^+ & -c^- \\ -c^- & -s^+ & c^+ & s^- \\ s^+ & -c^- & -s^- & c^+ \end{pmatrix}:=\mathbf{M}. \]
Solving \eqref{IM3} via Duhamel's formula leads to
\begin{equation}
\begin{pmatrix} C_V \\ S_V \\ C_N \\ S_N \end{pmatrix} (y) = 
\mathbf{M}(y) \mathbf{X}_0 + \varepsilon \kappa \mathbf{M}(y)\int_0^y Q(z) \begin{pmatrix} -c_\gamma^+ S_1 - s_\gamma^- C_1 \\[2pt] -s_\gamma^- S_1 +c_\gamma^+ C_1 \\[2pt] c_\gamma^- S_1 + s_\gamma^+ C_1 \\[2pt] s_\gamma^+ S_1 - c_\gamma^- C_1 \end{pmatrix} (z) \, \text{d}z 
\label{Duha1}
\end{equation}
where $\mathbf{X}_0 \in \R^4$ is some constant vector. Since $Q$, $S_1$ and $C_1$ are $L^2$ functions, and trigonometric functions are bounded, the integral on the right-hand side converges. Studying \eqref{Duha1} when $y \to + \infty$ and knowing that $C_V$, $S_V$, $C_N$ and $S_N$ belong to $L^2$, it follows that
\[ \mathbf{X}_0 + \varepsilon \kappa\int_0^{+ \infty} Q(z) \begin{pmatrix} -c_\gamma^+ S_1 - s_\gamma^- C_1 \\[2pt] -s_\gamma^- S_1 +c_\gamma^+ C_1 \\[2pt] c_\gamma^- S_1 + s_\gamma^+ C_1 \\[2pt] s_\gamma^+ S_1 - c_\gamma^- C_1 \end{pmatrix} (z) \, \text{d}z = 0. \]
This leads to
\begin{equation}\label{Duha2}
\begin{pmatrix} C_V \\ S_V \\ C_N \\ S_N \end{pmatrix} (y) = 
- \varepsilon \kappa \mathbf{M} (y) \int_y^{+ \infty} Q(z) \begin{pmatrix} -c_\gamma^+ S_1 - s_\gamma^- C_1 \\[2pt] -s_\gamma^- S_1 +c_\gamma^+ C_1 \\[2pt] c_\gamma^- S_1 + s_\gamma^+ C_1 \\[2pt] s_\gamma^+ S_1 - c_\gamma^- C_1 \end{pmatrix} (z) \, \text{d}z.
\end{equation}

\subsection{The eigenvalue zero case}\label{S:3.3} 
Assume that $\lambda = 0$ so that $\varepsilon = 0$ and \eqref{Duha2} gives $C_V=S_V=C_N=S_N=0$. Hence, by \eqref{IM1}, $L_-C_2=L_+S_1=L_-S_2=L_+C_1=0$, which leads to $C_2,S_2 \in \ker L_- = \SPAN (Q)$ and $C_1,S_1 \in \ker L_+ = \SPAN (Q')$. This is the first case in Theorem \ref{thmIM}.

\medskip

From now on, we assume $\lambda \neq 0$. Since $(\lambda , C_1 , C_2 , C_N , C_V , S_1 , S_2 , S_N , S_V )$ is a solution of \eqref{IM1} if and only if $(- \lambda , C_1 , C_2 , C_N , C_V , - S_1 , - S_2 , -S_N , -S_V )$ is a solution of \eqref{IM1}, possibly replacing $\lambda$ by $- \lambda$ and $S_*$ by $-S_*$, we also assume without loss of generality that $\lambda > 0$.

\subsection{Additional almost orthogonality relations}\label{S:3.4}

Using the identities \eqref{IMorthos} and \eqref{Duha2}, as well as the system \eqref{IM1}, 
we estimate certain scalar products involving the functions $S_1$, $C_1$, $S_2$, $C_2$
and related to the coercivity properties stated in \eqref{4.spec}.

First, since $L_-C_2 = \lambda S_1$, $L_- S_2 = - \lambda C_1$ and $L_- Q=0$, one has 
readily 
\begin{equation}\label{o1}
\langle S_1 , Q \rangle = \langle C_1 , Q \rangle = 0.
\end{equation}
Second, using \eqref{IM1} and the identity $L_+ ( \Lambda Q ) = -Q$, it follows that
\[
\langle C_2 , \Lambda Q \rangle = -\lambda^{-1} \langle S_N , Q \Lambda Q \rangle 
\]
and so by the Cauchy-Schwarz inequality,
\begin{equation}
\left| \langle C_2 , \Lambda Q \rangle \right| \lesssim \lambda^{-1} \left( \int_{\R} QS_N^2 \right)^\frac12. \label{estim1C2}
\end{equation}
Similarly,
\begin{equation}
\left| \langle S_2 , \Lambda Q \rangle \right| \lesssim \lambda^{-1} \left( \int_{\R} QC_N^2 \right)^\frac12.
\label{estim1S2}
\end{equation}
Using Duhamel's formula \eqref{Duha2}, it is clear that
\begin{equation}\label{estim1DT}
| C_V | + |S_V | + | C_N | + |S_N | \lesssim \varepsilon \int_\R Q \left( |S_1 | + |C_1 | \right)
\lesssim \varepsilon \left( \int_{\R} QS_1^2 +\int_{\R} QC_1^2 \right)^\frac12 . 
\end{equation}
Therefore,
\begin{equation}
\int_{\R} QC_V^2 + \int_{\R} QS_V^2 + \int_{\R} QC_N^2 + \int_{\R} QS_N^2 \lesssim \varepsilon^2 \left(\int_\R QS_1^2 + \int_{\R} QC_1^2 \right).
\label{estim2DT}
\end{equation}
Combining \eqref{estim1C2}, \eqref{estim1S2} and \eqref{estim2DT}, it follows that
\begin{equation}
\left| \langle C_2 , \Lambda Q \rangle \right| + \left| \langle S_2 , \Lambda Q \rangle \right| \lesssim \varepsilon \lambda^{-1} \left( \int_{\R} QS_1^2 + \int_{\R} QC_1^2 \right)^\frac12 . 
\label{estim2SC}
\end{equation}
Then, taking a suitable linear combinaison of the identities \eqref{IMorthos} and using 
the third line of~\eqref{IM1}, we obtain
\begin{equation} \label{r1}
\begin{aligned}
& \int_{\R} \left( (1+\beta ) \sin \left( \frac{\varepsilon y}{1+\beta} \right) - (1-\beta ) \sin \left( \frac{\varepsilon y}{1-\beta} \right) \right) QS_1 \, \text{d}y \\
&= \int_{\R} \left( (1+\beta) \cos \left( \frac{\varepsilon y}{1+\beta} \right) + (1-\beta ) \cos \left( \frac{\varepsilon y}{1-\beta} \right) \right) QC_1 \, \text{d}y \\
&= - \lambda^{-1} \int_{\R} S_2 \Psi
\end{aligned}
\end{equation}
where
\[
\Psi := L_- \left( (1+ \beta) \cos \left( \frac{\varepsilon y}{1+\beta} \right) Q
+ (1-\beta ) \cos \left( \frac{\varepsilon y}{1-\beta} \right) Q \right).
\]
By $L_- Q=0$, we check that
\begin{align*}
\Psi& = 2 \varepsilon Q' \left( \sin \left( \frac{\varepsilon y}{1+\beta} \right) + \sin \left( \frac{\varepsilon y}{1- \beta} \right) \right) \\
&\quad + \varepsilon^2 Q \left( \frac{1}{1+\beta} \cos \left( \frac{\varepsilon y}{1+\beta} \right) + \frac{1}{1-\beta} \cos \left( \frac{\varepsilon y}{1-\beta} \right) \right) .
\end{align*}
On the one hand, using the estimates
\begin{align*}
& \left| \left( \sin \left( \frac{\varepsilon y}{1+\beta} \right) + \sin \left( \frac{\varepsilon y}{1-\beta} \right) \right) - \frac{2 \varepsilon y}{1-\beta^2} \right| \lesssim \varepsilon^3 (1+|y|^3) \\
& \left| \left( \frac{1}{1+\beta} \cos \left( \frac{\varepsilon y}{1+\beta} \right) + \frac{1}{1-\beta} \cos \left( \frac{\varepsilon y}{1-\beta} \right) \right) - \frac{2}{1-\beta^2} \right| \lesssim \varepsilon^2 (1+y^2)
\end{align*}
we obtain
\begin{equation}
\left| \Psi - \frac{2 \varepsilon^2}{1-\beta^2} ( 2yQ'+Q) \right| \lesssim \varepsilon^4 (1+|y|^3) Q(y).
\label{r2}
\end{equation}
On the other hand,
\[
\left| (1+\beta ) \sin \left( \frac{\varepsilon y}{1+\beta} \right) - (1-\beta ) \sin \left( \frac{\varepsilon y}{1-\beta} \right) \right| \lesssim \varepsilon^3 (1+|y|^3),
\]
and thus by \eqref{r1},
\begin{equation}\label{rr1}
\left| \int_{\R} S_2 \Psi \right|
\lesssim \lambda \varepsilon^3 \int_{\R} (1+|y|^3) Q(y) |S_1(y)| \, \text{d}y
\lesssim\lambda \varepsilon^3 \left( \int_{\R} QS_1^2 \right)^\frac12.
\end{equation}
Combining \eqref{r2} and \eqref{rr1}, it follows that
\begin{align}
 \left| \int_{\R} (2yQ'+Q) S_2 \right| 
& \lesssim \varepsilon^{-2} \left| \int_{\R} S_2 \Psi \right| + C \varepsilon^2 \int_{\R} (1+|y|^3) Q(y) |S_2(y)| \, \text{d}y \nonumber\\
&\lesssim \varepsilon \lambda\left( \int_{\R} QS_1^2 \right)^\frac12+ \varepsilon^2 \left( \int_{\R} QS_2^2 \right)^\frac12
\label{estim3S2}
\end{align}
Similarly, using \eqref{IMorthos}, it holds
\begin{equation}
\left| \int_{\R} (2yQ' + Q) C_2 \right| \lesssim \varepsilon \lambda \left( \int_{\R} QC_1^2 \right)^\frac12 + \varepsilon^2 \left( \int_{\R} QC_2^2 \right)^\frac12.
\label{estim3C2}
\end{equation}
Gathering \eqref{estim2SC}, \eqref{estim3S2} and \eqref{estim3C2}, we get
\begin{equation}\label{estim4C2}
\begin{split}
\left| \langle S_2 , Q \rangle \right| +\left| \langle C_2 , Q \rangle \right| &
 \lesssim  \varepsilon ( \lambda + \lambda^{-1} ) \left( \int_{\R} QS_1^2 +\int_{\R} QC_1^2 \right)^\frac12 \\
& \quad + \varepsilon^2 \left( \int_{\R} QS_2^2 + \int_{\R} QC_2^2 \right)^\frac12.
\end{split}
\end{equation}
Different choices of linear combination in \eqref{IMorthos} give other estimates, with similar proofs 
and using the relation $L_-(yQ)=-2Q'$. For example, the identity
\begin{align*}
&\int_{\R} \left( \cos \left( \frac{\varepsilon y}{1+\beta} \right) - \cos \left( \frac{\varepsilon y}{1-\beta} \right) \right) QC_1 \, \text{d}y \\
&\quad = \int_{\R} \left( \sin \left( \frac{\varepsilon y}{1+\beta} \right) + \sin \left( \frac{\varepsilon y}{1-\beta} \right) \right) QS_1 \, \text{d}y
\end{align*}
leads to
\begin{equation}
\left| \langle C_2 , Q' \rangle \right| \lesssim \varepsilon \lambda \left( \int_{\R} QC_1^2 \right)^\frac12 + \varepsilon^2 \left( \int_{\R} QC_2^2 \right)^\frac12,
\label{estim5C2}
\end{equation}
while the identity
\begin{align*}
&\int_{\R} \left( \cos \left( \frac{\varepsilon y}{1+\beta} \right) - \cos \left( \frac{\varepsilon y}{1-\beta} \right) \right) QS_1 \, \text{d}y \\
&\quad = - \int_{\R} \left( \sin \left( \frac{\varepsilon y}{1+\beta} \right) + \sin \left( \frac{\varepsilon y}{1-\beta} \right) \right) QC_1 \, \text{d}y
\end{align*}
yields
\begin{equation}
\left| \langle S_2 , Q' \rangle \right| \lesssim \varepsilon \lambda \left( \int_{\R} QS_1^2 \right)^\frac12 + \varepsilon^2 \left( \int_{\R} QS_2^2 \right)^\frac12.
\label{estim5S2}
\end{equation}

In the next two subsections we show that $\lambda$ and $\lambda^{-1}$ are bounded regardless of~$\omega$, 
starting with an upper bound.

\subsection{Uniform upper bound on the eigenvalue}\label{S:3.up}
In this subsection, using Poho\-zaev-type arguments, we shall prove that $\lambda$ is bounded regardless of $\omega$. Fix a smooth even function $\chi \, : \, \R \to \R$ satisfying $\chi \equiv 1$ on $[0 \, , 1]$, $\chi \equiv 0$ on $[2 \, , + \infty )$ and $\chi ' \leqslant 0$ on $[0 \, , + \infty )$. For $A \gg 1$, introduce
\[ \zeta_A (y) = \exp \left( - \frac{|y|}{A} (1- \chi (y)) \right) \qquad \text{and} \qquad \Phi_A (y) = \int_0^y \zeta_A^2 . \]
Note that $\Phi_A ' = \zeta_A^2$ and that $|\Phi_A|\leqslant |y|$, as $0\leqslant \zeta_A\leqslant 1$ on $\R$.
Moreover, as $A \to + \infty$, $\zeta_A (y) \to 1$ and $\Phi_A (y) \to y$.
 We shall need the simple lemma below.

\begin{lemma} \label{lemA}
Let $\delta >0$. For $A>0$ large enough (depending on $\delta$ and $\lambda$),
for any $w \in H^2 ( \R )$,
\[ \| \zeta_A w' \|^2 \lesssim \delta \lambda^2 \| \zeta_A w \|^2 + \frac{1}{\delta \lambda^2} \| \zeta_A w'' \|^2. \]
\end{lemma}

\begin{proof}
The estimate $| ( \zeta_A^2 )' | \lesssim \frac{1}{A} \zeta_A^2$
is a direct consequence of the definition of $\zeta_A$. Thus, integrating by parts and then using Cauchy-Schwarz and Young inequalities,
\begin{align*}
\int_{\R} \zeta_A^2 (w')^2 &= - \int_{\R} ( \zeta_A^2) ' w w' - \int_{\R} \zeta_A^2 ww'' \\
& \lesssim \frac{1}{A} \int_{\R} \zeta_A^2 |ww'| + \int_{\R} \zeta_A^2 |ww''| \\
& \lesssim \frac{1}{A} \left ( \| \zeta_A w \|^2 + \| \zeta_A w' \|^2 \right ) + \delta \lambda^2 \| \zeta_A w \|^2 + \frac{1}{\delta \lambda^2} \| \zeta_A w'' \|^2 .
\end{align*}
Taking $A>0$ large enough (depending on $\delta$ and $\lambda$), the result follows. 
\end{proof}

We introduce the auxilliary functions
\begin{equation}\label{up0}
\begin{aligned} 
& \widetilde{S}_N = S_N + 2QS_1 , & & \widetilde{S}_V = S_V + 2 \beta QS_1 , \\
& \widetilde{C}_N = C_N + 2QC_1  , & & \widetilde{C}_V = C_V + 2 \beta QC_1.
\end{aligned}
\end{equation}
From \eqref{IM1}-\eqref{IM2} we see that the functions $(S_1 , C_1 , S_2 , C_2 , \widetilde{S}_N , \widetilde{C}_N , \widetilde{S}_V , \widetilde{C}_V )$ satisfy
\begin{equation}
\left \{ \begin{array}{l}
L_- C_2 = \lambda S_1 \\
L_- S_1 = \lambda C_2 + Q \widetilde{S}_N \\
L_- S_2 = - \lambda C_1 \\
L_- C_1 = - \lambda S_2 + Q \widetilde{C}_N
\end{array} \right.
\label{up1}
\end{equation}
and
\begin{equation}
\left \{ \begin{array}{l}
\widetilde{C}_V' - \beta \widetilde{C}_N' = - \varepsilon \widetilde{S}_N \\
\widetilde{S}_V' - \beta \widetilde{S}_N' = \varepsilon \widetilde{C}_N \\
\widetilde{C}_N' - \beta \widetilde{C}_V' = - \varepsilon \widetilde{S}_V + 2 \nu^2 (QC_1)' \\
\widetilde{S}_N' - \beta \widetilde{S}_V' = \varepsilon \widetilde{C}_V + 2 \nu^2 (QS_1)'
\end{array} \right.
\label{up2}
\end{equation}
where $\nu=\sqrt{1-\beta^2}$.

From \eqref{up1}, we have (note the convenient cancellation $L_-Q=0$ which avoids the presence of
the functions $\widetilde S_N$ and $\widetilde C_N$ without derivative on the right-hand side)
\begin{align}
& L_-^2S_1 - \lambda^2 S_1 = -2Q' \widetilde{S}_N' - Q \widetilde{S}_N'' \label{up3} \\
\text{and} \quad & L_-^2 C_1 - \lambda^2 C_1 = -2Q' \widetilde{C}_N' - Q \widetilde{C}_N'' . \label{up4}
\end{align}
We compute
\[ L_-^2 = \partial_y^4 + \mathbf{P}_2 \partial_y^2 + \mathbf{P}_1 \partial_y + \mathbf{P}_0 \]
where $\mathbf{P}_2 = -2(1-Q^2)$, $\mathbf{P}_1 = 4QQ'$ and $\mathbf{P}_0 = (1-Q^2)^2+(Q^2)''$. 
Note that
\begin{equation} 
| \mathbf{P}_0 | + | \mathbf{P}_2 | \lesssim 1 , \quad | \mathbf{P}_0^{(j)} | + | \mathbf{P}_2^{(j)} | \lesssim Q \quad \text{and} \quad | \mathbf{P}_1^{(k)} | \lesssim Q 
\label{upP}
\end{equation}
for any $j \geqslant 1$ and $k \geqslant 0$. Now we multiply \eqref{up3} by $\Phi_A S_1'$ and integrate. Integrating by parts, we see that
\begin{align*}
& \int_{\R} \Phi_A S_1' S_1'''' = \frac{3}{2} \int_{\R} \zeta_A^2 (S_1'')^2 - \frac{1}{2} \int_{\R} ( \zeta_A^2 )'' (S_1')^2 , \\
& \int_{\R} \Phi_A S_1' \mathbf{P}_2 S_1 '' = - \frac{1}{2} \int_{\R} ( \Phi_A \mathbf{P}_2 ' + \zeta_A^2 \mathbf{P}_2 ) (S_1')^2 \\
\text{and} \quad & \int_{\R} \Phi_A S_1' \mathbf{P}_0 S_1 = - \frac{1}{2}  \int_{\R} ( \Phi_A \mathbf{P}_0' + \zeta_A^2 \mathbf{P}_0 ) S_1^2 .
\end{align*}
Thus the left-hand side of \eqref{up3} gives
\[ \int_{\R} \Phi_A S_1' (L_-^2S_1-\lambda^2S_1) = \frac{3}{2} \int_{\R} \zeta_A^2 (S_1'')^2 + \int_{\R} \mathbf{R}_{A,1} (S_1')^2 + \int_{\R} \mathbf{R}_{A,0} S_1^2 + \frac{\lambda^2}{2} \int_{\R} \zeta_A^2 S_1^2 \]
where $\mathbf{R}_{A,1} = \Phi_A \left ( \mathbf{P}_1 - \frac{1}{2} \mathbf{P}_2' \right ) - \frac{1}{2} \zeta_A^2 \mathbf{P}_2 - \frac{1}{2} ( \zeta_A^2 )''$ and $\mathbf{R}_{A,0} = - \frac{1}{2} ( \Phi_A \mathbf{P}_0' + \zeta_A^2 \mathbf{P}_0)$. Therefore, from \eqref{up3} and then by integration by parts, we get
\begin{align*}
& \frac{3}{2} \int_{\R} \zeta_A^2 (S_1'')^2 + \int_{\R} \mathbf{R}_{A,1} (S_1')^2 + \int_{\R} \mathbf{R}_{A,0} S_1^2 + \frac{\lambda^2}{2} \int_{\R} \zeta_A^2 S_1^2 \\
& \quad  = \int_{\R} \Phi_A S_1' (-2Q' \widetilde{S}_N' - Q \widetilde{S}_N'') \\
& \quad = \int_{\R} (\zeta_A^2 Q - \Phi_A Q') S_1' \widetilde{S}_N' + \int_{\R} \Phi_A Q S_1'' \widetilde{S}_N'.
\end{align*}
From \eqref{up4}, we get similarly
\begin{align*}
& \frac{3}{2} \int_{\R} \zeta_A^2 (C_1'')^2 + \int_{\R} \mathbf{R}_{A,1} (C_1')^2 + \int_{\R} \mathbf{R}_{A,0} C_1^2 + \frac{\lambda^2}{2} \int_{\R} \zeta_A^2 C_1^2 \\
& \quad = \int_{\R} (\zeta_A^2 Q - \Phi_A Q') C_1' \widetilde{C}_N' + \int_{\R} \Phi_A Q C_1'' \widetilde{C}_N' .
\end{align*}
Summing these two identities, we obtain
\begin{equation}
\begin{split}
& \frac{3}{2} \left ( \| \zeta_A S_1'' \|^2 + \| \zeta_A C_1'' \|^2 \right )
+ \frac{\lambda^2}{2} \left ( \| \zeta_A S_1 \|^2 + \| \zeta_A C_1 \|^2 \right )
\\
& \quad+ \int_{\R} \mathbf{R}_{A,1} ((S_1')^2+(C_1')^2) + \int_{\R} \mathbf{R}_{A,0} (S_1^2 + C_1^2)  \\
& \quad = \int_{\R} ( \zeta_A^2 Q - \Phi_A Q') (S_1' \widetilde{S}_N '+C_1' \widetilde{C}_N ') + \int_{\R} \Phi_A Q (S_1'' \widetilde{S}_N ' + C_1'' \widetilde{C}_N ') .
\end{split}
\label{up5}
\end{equation}

Now, we use \eqref{up2} to obtain the following system for $(\widetilde S_N,\widetilde C_N)$
\begin{align}
& \nu^2 ( \widetilde{S}_N '' - 2(QS_1)'') = - \varepsilon^2 \widetilde{S}_N + 2 \beta \varepsilon \widetilde{C}_N' \label{up6} \\
\text{and} \quad & \nu^2 ( \widetilde{C}_N '' -2 (QC_1)'' ) = - \varepsilon^2 \widetilde{C}_N - 2 \beta \varepsilon \widetilde{S}_N ' . \label{up7}
\end{align}
We multiply \eqref{up6} by $\Phi_A \widetilde{S}_N'$ and integrate. By integration by parts, it follows that
\[
- \frac{\nu^2}{2} \int_{\R} \zeta_A^2 ( \widetilde{S}_N')^2 - 2 \nu^2 \int_{\R} \Phi_A (QS_1)'' \widetilde{S}_N ' \  = \frac{\varepsilon^2}{2} \int_{\R} \zeta_A^2 \widetilde{S}_N^2 + 2 \beta \varepsilon \int_{\R} \Phi_A \widetilde{S}_N ' \widetilde{C}_N'.
\]
Similarly, using \eqref{up7},
\[
- \frac{\nu^2}{2} \int_{\R} \zeta_A^2 ( \widetilde{C}_N')^2 - 2 \nu^2 \int_{\R} \Phi_A (QC_1)'' \widetilde{C}_N '   = \frac{\varepsilon^2}{2} \int_{\R} \zeta_A^2 \widetilde{C}_N^2 - 2 \beta \varepsilon \int_{\R} \Phi_A \widetilde{S}_N ' \widetilde{C}_N' .
\]
Summing the two identities above, we obtain (note a convenient cancellation)
\begin{align}
& \frac{\nu^2}{2} \left ( \| \zeta_A \widetilde{S}_N' \|^2 + \| \zeta_A \widetilde{C}_N ' \|^2 \right ) + \frac{\varepsilon^2}{2} \left ( \| \zeta_A \widetilde{S}_N \|^2 + \| \zeta_A \widetilde{C}_N  \|^2 \right ) \nonumber \\
& \quad = -2 \nu^2 \int_{\R} \Phi_A ((QS_1)'' \widetilde{S}_N' + (QC_1)'' \widetilde{C}_N ') \nonumber \\
& \quad = -2 \nu^2 \int_{\R} \Phi_A Q ( S_1'' \widetilde{S}_N' + C_1'' \widetilde{C}_N') -4 \nu^2 \int_{\R} \Phi_A Q' ( S_1' \widetilde{S}_N' + C_1' \widetilde{C}_N' ) \nonumber \\
& \qquad  -2 \nu^2 \int_{\R} \Phi_A Q'' (S_1 \widetilde{S}_N' + C_1 \widetilde{C}_N' ) . \label{up8}
\end{align}
We now combine \eqref{up5} and \eqref{up8} in order to make the term
$\int_{\R} \Phi_A Q ( S_1'' \widetilde{S}_N' + C_1'' \widetilde{C}_N')$ disappear. 
Explicitly, \eqref{up5}$+ \frac{1}{2 \nu^2} \times$\eqref{up8} gives
\begin{equation}
\begin{split}
& \frac{3}{2} \left ( \| \zeta_A S_1'' \|^2 + \| \zeta_A C_1'' \|^2 \right )  + \frac{\lambda^2}{2} \left ( \| \zeta_A S_1 \|^2 + \| \zeta_A C_1 \|^2 \right ) \\
& \quad + \frac{1}{4} \left ( \| \zeta_A \widetilde{S}_N' \|^2 + \| \zeta_A \widetilde{C}_N ' \|^2 \right ) + \frac{\varepsilon^2}{4 \nu^2} \left ( \| \zeta_A \widetilde{S}_N \|^2 + \| \zeta_A \widetilde{C}_N  \|^2 \right ) \\
& \quad = 
- \int_{\R} \mathbf{R}_{A,1} ((S_1')^2+(C_1')^2) - \int_{\R} \mathbf{R}_{A,0} (S_1^2 + C_1^2)\\
&\qquad +\int_{\R} ( \zeta_A^2 Q - 3 \Phi_A Q') ( S_1' \widetilde{S}_N' + C_1' \widetilde{C}_N' ) - \int_{\R} \Phi_A Q'' ( S_1 \widetilde{S}_N' + C_1 \widetilde{C}_N' ) .
\end{split}
\label{up9}
\end{equation}

Now we control the right-hand side of \eqref{up9}. First, using \eqref{upP}, we see that $| \mathbf{R}_{A,1} | \lesssim |y|Q + \zeta_A^2 \lesssim Q^{1/2} + \zeta_A^2 \lesssim \zeta_A^2$. Thus,
\begin{align}
& \left | \int_{\R} \mathbf{R}_{A,1} ((S_1')^2+(C_1')^2) \right | \nonumber \\
& \quad \lesssim \| \zeta_A S_1' \|^2 + \| \zeta_A C_1' \|^2 \nonumber \\
& \quad \lesssim \delta \lambda^2 \left ( \| \zeta_A S_1 \|^2 + \| \zeta_A C_1 \|^2 \right ) + \frac{1}{\delta \lambda^2} \left ( \| \zeta_A S_1'' \|^2 + \| \zeta_A C_1 '' \|^2 \right ) \label{up10}
\end{align}
using Lemma \ref{lemA}, where $\delta >0$ is a small parameter which we shall fix later.

Then, using \eqref{upP} again, we see that $| \mathbf{R}_{A,0} | \lesssim |y|Q + \zeta_A^2 \lesssim \zeta_A^2$, and so
\begin{equation} 
\left | \int_{\R} \mathbf{R}_{A,0} ( S_1^2 + C_1^2 ) \right | \lesssim \| \zeta_A S_1 \|^2 + \| \zeta_A C_1 \|^2 . 
\label{up11}
\end{equation}
Next, using Young's inequality and Lemma \ref{lemA},
\begin{align}
& \left | \int_{\R} ( \zeta_A^2 Q - 3 \Phi_A Q') ( S_1 ' \widetilde{S}_N ' + C_1 ' \widetilde{C}_N ' ) \right | \nonumber \\
& \quad \lesssim \int_{\R} \zeta_A^2 \left ( |S_1' \widetilde{S}_N ' | + | C_1 ' \widetilde{C}_N ' | \right ) \nonumber \\
& \quad \lesssim \| \zeta_A S_1 ' \| \, \| \zeta_A \widetilde{S}_N ' \| + \| \zeta_A C_1 ' \| \, \| \zeta_A \widetilde{C}_N ' \| \nonumber \\
& \quad \lesssim \sqrt{\lambda} \left ( \| \zeta_A S_1' \|^2 + \| \zeta_A C_1' \|^2 \right ) + \frac{1}{\sqrt{\lambda}} \left ( \| \zeta_A \widetilde{S}_N ' \|^2 + \| \zeta_A \widetilde{C}_N' \|^2 \right ) \nonumber \\
& \quad \lesssim \sqrt{\lambda} \left ( \lambda \| \zeta_A S_1 \|^2 + \frac{1}{\lambda} \| \zeta_A S_1 '' \|^2 + \lambda \| \zeta_A C_1 \|^2 + \frac{1}{\lambda} \| \zeta_A C_1 '' \|^2 \right ) \nonumber \\
& \qquad  + \frac{1}{\sqrt{\lambda}} \left ( \| \zeta_A \widetilde{S}_N ' \|^2 + \| \zeta_A \widetilde{C}_N \|^2 \right ) ,\nonumber
\end{align}
which we rewrite as
\begin{align}
& \left | \int_{\R} ( \zeta_A^2 Q - 3 \Phi_A Q') ( S_1 ' \widetilde{S}_N ' + C_1 ' \widetilde{C}_N ' ) \right | \nonumber \\
& \quad \lesssim \lambda\sqrt{\lambda} \left ( \| \zeta_A S_1 \|^2 + \| \zeta_A C_1 \|^2 \right ) + \frac{1}{\sqrt{\lambda}} \left ( \| \zeta_A S_1 '' \|^2 + \| \zeta_A C_1 '' \|^2 \right ) \label{up12}\\
& \qquad  + \frac{1}{\sqrt{\lambda}} \left ( \| \zeta_A \widetilde{S}_N ' \|^2 + \| \zeta_A \widetilde{C}_N ' \|^2 \right ) .  \nonumber
\end{align}
Finally, again by Young's inequality,
\begin{align}
& \left | \int_{\R} \Phi_A Q'' (S_1 \widetilde{S}_N' + C_1 \widetilde{C}_N') \right | \nonumber \\
& \quad \lesssim \int_{\R} \zeta_A^2 \left ( |S_1 \widetilde{S}_N' | + |C_1 \widetilde{C}_N' | \right ) \nonumber \\
& \quad \lesssim \| \zeta_A S_1 \| \, \| \zeta_A \widetilde{S}_N ' \| + \| \zeta_A C_1 \| \, \| \zeta_A \widetilde{C}_N ' \| \nonumber \\
& \quad \lesssim \lambda \left ( \| \zeta_A S_1 \|^2 + \| \zeta_A C_1 \|^2 \right ) + \frac{1}{\lambda} \left ( \| \zeta_A \widetilde{S}_N ' \|^2 + \| \zeta_A \widetilde{C}_N ' \|^2 \right ) . \label{up13}
\end{align}

Injecting \eqref{up10}, \eqref{up11}, \eqref{up12} and \eqref{up13} in \eqref{up9}, it follows that
for a constant $C>0$,
\begin{align*}
 & \left ( \frac{3}{2} - \frac{C}{\delta \lambda^2} - \frac{C}{\sqrt{\lambda}} \right ) \left ( \| \zeta_A S_1'' \|^2 + \| \zeta_A C_1'' \|^2 \right ) \\
&  \quad + \lambda^2 \left ( \frac{1}{2} - \frac{C}{\lambda^2} - C \delta - \frac{C}{\sqrt{\lambda}} - \frac{C}{\lambda} \right ) \left ( \| \zeta_A S_1 \|^2 + \| \zeta_A C_1 \|^2 \right ) \\
&  \quad + \left ( \frac{1}{4} - \frac{C}{\sqrt{\lambda}}- \frac{C}{\lambda} \right ) \left ( \| \zeta_A \widetilde{S}_N' \|^2 + \| \zeta_A \widetilde{C}_N ' \|^2 \right ) + \frac{\varepsilon^2}{4 \nu^2} \left ( \| \zeta_A \widetilde{S}_N \|^2 + \| \zeta_A \widetilde{C}_N \|^2 \right ) \leq 0 .
\end{align*}
Fix $\delta >0$ small enough (but independent of $\lambda$ and $\omega$) such that $\frac{1}{2} - C  \delta \geqslant \frac{1}{4}$. Now assume that $\lambda$ is larger than a certain constant (which depends on the constant  $C$  and  on $\delta$, but does not depend on $\omega$), such that
\[ \frac{3}{2} - \frac{C}{\delta \lambda^2} - \frac{C}{\sqrt{\lambda}} \geqslant 1 , \quad \frac{1}{4} - \frac{C}{\lambda^2} - \frac{C}{\sqrt{\lambda}} - \frac{C}{\lambda} \geqslant \frac{1}{8} \quad \text{and} \quad \frac{1}{4} - \frac{C}{\sqrt{\lambda}} -\frac{C}\lambda\geqslant \frac{1}{8} . \]
Therefore, taking $A>0$ large enough (depending on $\delta$ and $\lambda$), it follows that
\begin{align*} 
& \left ( \| \zeta_A S_1'' \|^2 + \| \zeta_A C_1'' \|^2 \right ) + \frac{\lambda^2}{8} \left ( \| \zeta_A S_1 \|^2 + \| \zeta_A C_1 \|^2 \right ) \\
&  \quad + \frac{1}{8} \left ( \| \zeta_A \widetilde{S}_N' \|^2 + \| \zeta_A \widetilde{C}_N ' \|^2 \right ) + \frac{\varepsilon^2}{4 \nu^2} \left ( \| \zeta_A \widetilde{S}_N \|^2 + \| \zeta_A \widetilde{C}_N \|^2 \right ) 
\leq 0.
\end{align*}
Hence $S_1=C_1=\widetilde{S}_N=\widetilde{C}_N=0$. From the second and fourth lines of \eqref{up1},
it follows that $ S_2= C_2=0$. 
From \eqref{up2}  it follows that $\widetilde S_V=\widetilde C_V=0$.
Lastly, from \eqref{up0}, $S_N=C_N=S_V=C_V=0$.

\medskip

In the remainder of the proof, we shall assume that $\lambda$ is uniformly bounded regardless of $\omega$, which means $\lambda \lesssim 1$.

\subsection{Uniform lower bound on the eigenvalue}\label{S:3.5}
Set
\[
\mathbf{H} := \|S_1\|_{H^1}^2 + \|C_1\|_{H^1}^2 + \|S_2\|_{H^1}^2 + \|C_2\|_{H^1}^2.
\]
By the first two lines of \eqref{IM1}, we see that
\begin{align*}
& \lambda \int_{\R} S_1 C_2 = \int_{\R} C_2 L_- C_2 ,\\
&\lambda \int_{\R} S_1 C_2 = \int_{\R} S_1 ( L_+ S_1 - Q S_N ) 
\end{align*}
and so, adding up these identities,
\begin{equation}
\langle L_+ S_1 , S_1 \rangle + \langle L_- C_2 , C_2 \rangle 
- \int_{\R} QS_1 S_N 
= 2 \lambda \int_{\R} S_1 C_2 . \label{4.1}
\end{equation}
In order to use \eqref{4.spec}, we control three scalar products
\begin{itemize}
\item First, from \eqref{o1}, $\langle S_1 , Q \rangle = 0$. 
\item Second, since $L_- (yQ)=-2Q'$ and $L_-C_2 = \lambda S_1$,
\[
\left| \langle S_1 , yQ \rangle \right| = \left| \lambda^{-1} \langle L_- C_2 , yQ \rangle \right| 
= \lambda^{-1} \left| \langle C_2 , -2Q' \rangle \right| 
\]
and so by \eqref{estim5C2}
\[
\left| \langle S_1 , yQ \rangle \right|
\lesssim \varepsilon \left( \int_{\R} QC_1^2 \right)^\frac12 + \lambda^{-1} \varepsilon^2 \left( \int_{\R} QC_2^2 \right)^\frac12.
\]
Recalling that $\varepsilon = \lambda \sqrt{\omega}$, we get the rough estimate
\[
\left| \langle S_1 , yQ \rangle \right| 
\lesssim \lambda \sqrt{\omega} \mathbf{H} ^\frac12.
\]
\item Third, \eqref{estim2SC} gives
\[
\left| \langle C_2 , \Lambda Q \rangle \right| \lesssim \varepsilon \lambda^{-1} \left( \int_{\R} QS_1^2+\int_{\R} QC_1^2 \right)^\frac12 \lesssim \sqrt{\omega} \mathbf{H}^\frac12.
\]
\end{itemize}
Gathering the three estimates above, we have proved that
\begin{equation}
\langle S_1 , Q \rangle^2 + \langle S_1 , yQ \rangle^2 + \langle C_2 ,\Lambda Q \rangle^2 \lesssim \omega \mathbf{H} .
\label{4.proj}
\end{equation}
Now, we control a scalar product in \eqref{4.1} via the Cauchy-Schwarz inequality and then \eqref{estim2DT}
\begin{align}
\left| \int_{\R} QS_1 S_N \right| & \lesssim \left( \int_{\R} QS_1^2 \right)^\frac12 \left( \int_{\R} QS_N^2 \right)^\frac12 \nonumber\\
&\lesssim\varepsilon \left( \int_{\R} QS_1^2 \right)^\frac12 \left( \int_{\R} QS_1^2 + \int_{\R} QC_1^2 \right)^\frac12 \lesssim \sqrt{\omega} \mathbf{H}. \label{4.2}
\end{align}
Combining \eqref{4.spec}, \eqref{4.1}, \eqref{4.proj} and \eqref{4.2}, we find that
\begin{align}
& C \left( \|S_1\|_{H^1}^2 + \|C_2\|_{H^1}^2 \right) - C' \sqrt{\omega} \, \mathbf{H} \nonumber \\
&\quad \leqslant \langle L_+ S_1 , S_1 \rangle + \langle L_- C_2 , C_2 \rangle - \int_{\R} QS_1S_N \nonumber\\
&\quad\leqslant 2 \lambda \int_{\R} S_1C_2
\leqslant\lambda \left( \|S_1\|_{H^1}^2 + \|C_2\|_{H^1}^2 \right) \label{4.est1}
\end{align}
By the same argument starting with the last two lines of \eqref{IM1}, we obtain similarly
\begin{equation}
C \left( \|S_2\|_{H^1}^2 + \|C_1\|_{H^1}^2 \right) - C' \sqrt{\omega} \, \mathbf{H}
\leqslant
\lambda \left( \|S_2\|_{H^1}^2 + \|C_1\|_{H^1}^2 \right). 
\label{4.est2}
\end{equation}
Summing \eqref{4.est1} and \eqref{4.est2}, we get $ C \mathbf{H} - C' \sqrt{\omega} \mathbf{H}\leqslant \lambda \mathbf{H}$.
Taking $\omega > 0$ small enough, we have $C\mathbf{H}\leqslant \lambda \mathbf{H} $.
If $\mathbf{H}=0$, then $C_1=S_1=C_2=S_2=0$ and it follows that also $C_V=S_V=C_N=S_N=0$ (see \eqref{Duha2} for example). 
In the remainder of the proof, we assume that $\mathbf{H} > 0$, which yields $\lambda \geqslant C$, where the constant $C>0$ is independent of
$\omega$.

\subsection{The transformed problem}\label{S:3.6}
As in some previous works (\cite{Ko1,Ko3,Ma1,GR1}),
we shall use a transformed problem based on the factorisation property \eqref{facto}.
We introduce 
\[
W_2 = S^2C_2\quad \mbox{and}\quad Z_2 = S^2S_2.
\] 
Since $C_2,S_2 \in H^s ( \R )$ for all $s \geqslant 0$, we also have $W_2,Z_2 \in H^s ( \R )$ for all $s \geqslant 0$. Using identity \eqref{facto} and then system \eqref{IM1}, it follows that
\begin{align*}
M^2 W_2 = M^2 S^2 C_2 
&= S^2L_+L_- C_2 \\
&= \lambda S^2 \left( \lambda C_2 + QS_N \right)
=\lambda^2 W_2 + \lambda S^2QS_N
\end{align*}
Note that by the definition of $S$, $S^2 (QS_N) = QS_N''$,
so that 
\begin{equation}
M^2 W_2 = \lambda^2 W_2 + F_W \quad \mbox{where} \quad
F_W := \lambda QS_N''.
\label{MW}
\end{equation}
Similarly, 
\begin{equation}
M^2 Z_2 = \lambda^2 Z_2 + F_Z
\quad \mbox{where}\quad F_Z := -\lambda QC_N''.
\label{MZ}
\end{equation}
Let us estimate $W_2 = S^2 C_2$ in terms of $C_2$. See \cite{Ma1} for similar estimates that we adapt here.
To begin with, using $W_2=S^2C_2$ and integrating, one obtains
\begin{equation}
C_2 = ayQ + bQ + Q \int_0^y \int_0^z \frac{W_2}{Q} 
\label{C2W2}
\end{equation}
for some integration constants $a,b \in \R$. 
First, we estimate $a$. Taking the scalar product of
equation \eqref{C2W2} by $Q'$, we have
\[ -a \underbrace{\langle yQ , Q' \rangle }_{= \, -2} = - \langle C_2 , Q' \rangle + b \underbrace{\langle Q , Q' \rangle }_{= \, 0} + \left\langle Q' , Q \int_0^y \int_0^z \frac{W_2}{Q} \right\rangle \]
and we estimate the last term on the right-hand side as follows
\begin{align*}
\left| \left\langle Q' , Q \int_0^y \int_0^z \frac{W_2}{Q} \right\rangle \right| \lesssim & \int_{\R} Q^2 \int_0^y \int_0^z \frac{|W_2|}{Q} \\
\lesssim & \int_{\R} Q^2 \left( \int_{\R} Q^{\frac{1}{2}} W_2^2 \right)^\frac12 \left( \int_0^y Q^{- \frac{5}{2}} \right)^\frac12 \\
\lesssim & \left( \int_{\R} Q^{\frac{1}{2}} W_2^2 \right)^\frac12 \int_{\R} Q^{\frac{3}{4}} \\
\lesssim & \left( \int_{\R} Q^{\frac{1}{2}} W_2^2 \right)^\frac12.
\end{align*}
Using also the estimate \eqref{estim5C2}, we obtain
\begin{equation}
|a| \lesssim \left( \int_{\R} Q^{\frac{1}{2}} W_2^2 \right)^\frac12 + \varepsilon \left( \int_{\R} QC_1^2 \right)^\frac12 + \varepsilon^2 \left( \int_{\R} QC_2^2 \right)^\frac12.
\label{estima}
\end{equation}
Now, we estimate $b$. Taking the scalar product of
equation \eqref{C2W2} by $Q$, we have
\[ b \underbrace{\langle Q , Q \rangle }_{= \, 4} = \langle C_2 , Q \rangle - a \underbrace{\langle yQ , Q \rangle }_{= \, 0} - \left\langle Q , Q \int_0^y \int_0^z \frac{W_2}{Q} \right\rangle . \]
Using \eqref{estim4C2}, we obtain
(recall that $\lambda + \lambda^{-1} \lesssim 1$)
\begin{equation}
\begin{split}
|b| & \lesssim \left( \int_{\R} Q^{\frac{1}{2}} W_2^2 \right)^\frac12 + \varepsilon \left( \int_{\R} QC_1^2 +\int_{\R} QS_1^2 \right)^\frac12 \\
& \quad + \varepsilon^2 \left( \int_{\R} QC_2^2 + \int_{\R} QS_2^2 \right)^\frac12.
\end{split}
\label{estimb}
\end{equation}
Therefore, using \eqref{C2W2} again, we find
\begin{align*}
\int_{\R} QC_2^2 \lesssim & \int_{\R} \left( a^2 y^2 Q^2 + b^2 Q^2 + Q^2 \left( \int_0^y \int_0^z \frac{W_2}{Q} \right)^2 \right) \\
\lesssim & \, \, a^2 + b^2 + \int_{\R} Q^{\frac{1}{2}} W_2^2 \\
\lesssim & \int_{\R} Q^{\frac{1}{2}} W_2^2 + \varepsilon^2 \int_{\R} QC_1^2 + \varepsilon^2 \int_{\R} QS_1^2 + \varepsilon^4 \int_{\R} QC_2^2 + \varepsilon^4 \int_{\R} Q S_2^2.
\end{align*}
Similarly, we have
\[ \int_{\R} QS_2^2 \lesssim \int_{\R} Q^{\frac{1}{2}} Z_2^2 + \varepsilon^2 \int_{\R} QC_1^2 + \varepsilon^2 \int_{\R} QS_1^2 + \varepsilon^4 \int_{\R} QC_2^2 + \varepsilon^4 \int_{\R} Q S_2^2. \]
Summing both estimates above and taking $\omega >0$ small enough, it follows that
\begin{equation}
\int_{\R} QC_2^2 + \int_{\R} QS_2^2 \lesssim \int_{\R} Q^{\frac{1}{2}} W_2^2 + \int_{\R} Q^{\frac{1}{2}} Z_2^2 + \varepsilon^2 \int_{\R} QC_1^2 + \varepsilon^2 \int_{\R} QS_1^2 .
\label{estim6C2}
\end{equation}
We also estimate the weighted norms of $C_2'$ and $S_2'$. Indeed, differentiating \eqref{C2W2},
\[ C_2 ' = a (yQ)' + bQ' + Q \int_0^y \frac{W_2}{Q} + Q' \int_0^y \int_0^z \frac{W_2}{Q}. \]
Using the estimates \eqref{estima} and \eqref{estimb}, and proceeding as above, we find that
\begin{align}
\int_{\R} Q((C_2')^2+(S_2')^2) \lesssim & \int_{\R} Q^{\frac{1}{2}} (W_2^2+Z_2^2) + \varepsilon^2 \int_{\R} Q(C_1^2+S_1^2) + \varepsilon^4 \int_{\R} Q(C_2^2+S_2^2) \nonumber \\
\lesssim & \int_{\R} Q^{\frac{1}{2}} (W_2^2+Z_2^2) + \varepsilon^2 \int_{\R} QC_1^2 + \varepsilon^2 \int_{\R} QS_1^2 \label{estimC2'}
\end{align}
using \eqref{estim6C2} for the last estimate. Similarly,
differentiating \eqref{C2W2} twice
\begin{equation}
\int_{\R} Q((C_2'')^2+(S_2'')^2) \lesssim \int_{\R} Q^{\frac{1}{2}} (W_2^2+Z_2^2) + \varepsilon^2 \int_{\R} QC_1^2 + \varepsilon^2 \int_{\R} QS_1^2 ,
\label{estimC2''}
\end{equation}
and differentiating thrice
\begin{equation}
\begin{split}
\int_{\R} Q((C_2''')^2+(S_2''')^2) & \lesssim 
\int_{\R} Q^{\frac{1}{2}} ((W_2')^2+(Z_2')^2) +
\int_{\R} Q^{\frac{1}{2}} (W_2^2+Z_2^2) \\
& \quad + \varepsilon^2 \int_{\R} QC_1^2 + \varepsilon^2 \int_{\R} QS_1^2 . 
\end{split}
\label{estimC2'''}
\end{equation}

These estimates enable us to control also weighted norms for $S_1, S_1', C_1,C_1'$.
Indeed, from \eqref{IM1} we have
\[ 
S_1^2 = \lambda^{-2} (L_-C_2)^2 = \lambda^{-2} ((1-Q^2) C_2 - C_2'')^2
\]
whence, using \eqref{estim6C2} and \eqref{estimC2''},
\begin{align*}
\int_{\R} QS_1^2 \lesssim & \, \, \lambda^{-2} \int_{\R} QC_2^2 + \lambda^{-2} \int_{\R} Q (C_2'')^2 \\
\lesssim & \, \, \int_{\R} Q^{\frac{1}{2}} W_2^2 + \varepsilon^2 \left( \int_{\R} QC_1^2 + \int_{\R} QS_1^2 \right),
\end{align*}
Similarly,
\[ \int_{\R} QC_1^2 \lesssim \int_{\R} Q^{\frac{1}{2}} Z_2^2 + \varepsilon^2 \left( \int_{\R} QC_1^2 + \int_{\R} QS_1^2 \right) \]
whence
\[ \int_{\R} QC_1^2 + \int_{\R} QS_1^2 \lesssim \int_{\R} Q^{\frac{1}{2}} W_2^2 + \int_{\R} Q^{\frac{1}{2}} Z_2^2 + \varepsilon^2 \left( \int_{\R} QC_1^2 + \int_{\R} QS_1^2 \right). \]
Taking $\omega >0$ small enough, it follows that
\begin{equation}
\int_{\R} QC_1^2 + \int_{\R} QS_1^2 \lesssim 
\int_{\R} Q^{\frac{1}{2}} (W_2^2 + Z_2^2) .
\label{estim7SC}
\end{equation}
Gathering \eqref{estim6C2} and \eqref{estim7SC}, it follows that
\begin{equation}
\int_{\R} Q \left( C_2^2 + (C_2')^2 + (C_2'')^2 + S_2^2 + (S_2')^2 + (S_2'')^2 \right) \lesssim \int_{\R} Q^{\frac{1}{2}} (W_2^2 + Z_2^2 ) .
\label{estim8SC}
\end{equation}
Using also \eqref{estimC2'''} and
following the same steps as for the proof of \eqref{estim7SC}, we get
\begin{equation}
\int_{\R} Q\left((C_1')^2+(S_1')^2 \right)
\lesssim \int_{\R} Q^{\frac{1}{2}}
\left((W_2')^2 + (Z_2')^2+ W_2^2 +Z_2^2 \right).
\label{estim7SC'}
\end{equation}

\subsection{Virial arguments on the transformed problem} \label{S:3.7}

The argument is adapted from the proof of the non-existence of NLS internal modes in \cite{Ma1,GR1}. 
However, we point out a major difference here. In the transformed
problem~\eqref{MW}, the operator $M^2$ is quite simple but it does not
have a potential. In \cite{Ma1,GR1}, the transformed problem
has a non-trivial repulsive potential, which 
happens to be crucial to prove the non-existence of internal mode.
Here, to compensate the lack of repulsive potential in the transformed operator,
we will use the additional almost orthogonality relations
\eqref{estim3S2} and \eqref{estim3C2}, obtained from 
the subsystem \eqref{IM3}
along with the coercivity inequality stated in Lemma \ref{lemh}.

We use localised virial arguments. We refer to \S\ref{S:3.up} for the definition of the functions $\Phi_A$ and $\zeta_A$. Multiply \eqref{MW} by $2 \Phi_A W_2' + \Phi_A ' W_2$ and integrate on $\R$.
Recalling that $M^2 = ( \partial_y^2 -1)^2$, we get
\begin{multline*}
\int_{\R} (W_2'''' -2 W_2'' + W_2)(2 \Phi_A W_2'+ \Phi_A ' W_2)\\
= \lambda^2 \int_{\R} W_2 (2 \Phi_A W_2'+ \Phi_A ' W_2) + \int_{\R} F_W (2 \Phi_A W_2'+ \Phi_A ' W_2).
\end{multline*}
Integrating by parts, we obtain the identities
\begin{align*}
\int_{\R} W_2''''(2 \Phi_A W_2'+ \Phi_A ' W_2) &= 4 \int_{\R} \zeta_A^2 (W_2'')^2 - 3 \int_{\R} (\zeta_A^2) '' (W_2')^2 + \frac{1}{2} \int_{\R} ( \zeta_A^2 )'''' W_2^2,\\
\int_{\R} W_2''(2 \Phi_A W_2'+ \Phi_A ' W_2) &= -2 \int_{\R} \zeta_A^2 (W_2')^2 + \frac{1}{2} \int_{\R} (\zeta_A^2)'' W_2^2,\\
\int_{\R} W_2(2 \Phi_A W_2'+ \Phi_A ' W_2)&=0.
\end{align*}
Hence,
\begin{multline}
4 \int_{\R} \zeta_A^2 (W_2'')^2 + 4 \int_{\R} \zeta_A^2 (W_2')^2 - 3 \int_{\R} ( \zeta_A^2 )'' (W_2')^2 \\
+ \int_{\R} \left( \frac{1}{2} ( \zeta_A^2 ) '''' - ( \zeta_A^2 )'' \right) W_2^2 = \int_{\R} F_W (2 \Phi_A W_2' + \Phi_A' W_2). 
\label{virielA1}
\end{multline}
We know that $W_2' , W_2'' \in L^2 ( \R )$ and $\zeta_A^2 (y) \to 1$ as $A \to + \infty$. Moreover, $\left| ( \zeta_A^2 )'' \right| \lesssim \frac{1}{A}$ on $\R$. Hence, by the dominated convergence Theorem,
\begin{align*}
& \int_{\R} \zeta_A^2 (W_2'')^2 \, \underset{A \to + \infty}{\longrightarrow} \, \int_{\R} (W_2'')^2 , \\
& \int_{\R} \zeta_A^2 (W_2')^2 \, \underset{A \to + \infty}{\longrightarrow} \, \int_{\R} (W_2')^2, \\
& \int_{\R} (\zeta_A^2)'' (W_2')^2 \, \underset{A \to + \infty}{\longrightarrow} \, 0.
\end{align*}
We could use the fact that $W_2 \in L^2$ and the estimate $\left| ( \zeta_A^2 )'' \right| + \left| (\zeta_A^2)'''' \right| \lesssim \frac{1}{A}$ to show that $\int_{\R} \left( \frac{1}{2} ( \zeta_A^2)'''' - ( \zeta_A^2 )'' \right) W_2^2 \, \underset{A \to + \infty}{\longrightarrow} \, 0$. 
However, anticipating the justification of Remark \ref{rk:2} in \S~\ref{prk}, 
we prefer to give a proof that relies only on the fact that $W_2 \in L^\infty ( \R )$, which is true here by the Sobolev injection $H^1(\R)\subset L^\infty(\R)$. Note that, on $\R$,
\[ \left| ( \zeta_A^2 )'' \right| + \left| ( \zeta_A^2 )'''' \right| \lesssim \frac{1}{A^2} \, e^{- \frac{2 |y|}{A}} + \frac{1}{A} \, \theta (y) \]
where $\theta$ is a smooth function that does not depend on $A$ and whose support satisfies $\text{supp} ( \theta ) \subset [-2 \, , 2]$. Therefore
\begin{align*}
\left| \int_{\R} \left( \frac{1}{2} ( \zeta_A^2 )'''' - ( \zeta_A^2 )'' \right) W_2^2 \right| & \lesssim ||W_2||_{L^\infty}^2 \left( \frac{1}{A^2} \int_{\R} e^{- \frac{2 |y|}{A}} \, \text{d}y + \frac{1}{A} || \theta ||_{L^1} \right) \\
& \lesssim \frac{1}{A} ||W_2||_{L^\infty}^2 \, \underset{A \to + \infty}{\longrightarrow} \, 0.
\end{align*}
Hence, gathering the convergence results above, the left-hand side of \eqref{virielA1} converges as follows:
\begin{multline*}
4 \int_{\R} \zeta_A^2 (W_2'')^2 + 4 \int_{\R} \zeta_A^2 (W_2')^2 - 3 \int_{\R} ( \zeta_A^2 )'' (W_2')^2 \\
+ \int_{\R} \left( \frac{1}{2} ( \zeta_A^2 ) '''' - ( \zeta_A^2 )'' \right) W_2^2 \, \underset{A \to + \infty}{\longrightarrow} \, 4 \int_{\R} (W_2'')^2 + 4 \int_{\R} (W_2')^2. 
\end{multline*}
Besides, by the Cauchy-Schwarz inequality and since $\Phi_A^2 \lesssim y^2$ and $( \Phi_A ')^2 \lesssim 1$, it holds that
\begin{align}
\left| \int_{\R} F_W ( 2 \Phi_A W_2' + \Phi_A ' W_2 ) \right| 
&\lesssim \left( \int_{\R} y^2 F_W^2 \right)^\frac12 \left( \int_{\R} (W_2')^2 \right)^\frac12 \nonumber \\ 
&\quad + \left( \int_{\R} Q^{- \frac{1}{2}} F_W^2 \right)^\frac12 \left( \int_{\R} Q^{\frac{1}{2}} W_2^2 \right)^\frac12.
\label{virielA2}
\end{align}
Taking the $\liminf$ as $A \to + \infty$ in \eqref{virielA1} and \eqref{virielA2} it follows that
\begin{align}
\int_{\R} (W_2'')^2 + \int_{\R} (W_2')^2 
&\lesssim \left( \int_{\R} y^2 F_W^2 \right)^\frac12 \left( \int_{\R} (W_2')^2 \right)^\frac12\nonumber\\ 
&\quad + \left( \int_{\R} Q^{- \frac{1}{2}} F_W^2 \right)^\frac12 \left( \int_{\R} Q^{\frac{1}{2}} W_2^2 \right)^\frac12.
\label{viriel1}
\end{align}
Now, we need to estimate $F_W = \lambda QS_N''$. We use \eqref{IM3} to compute $S_N''$ and find that
\begin{align*}
 F_W &=\varepsilon\lambda \gamma \kappa Q(Q'C_1+QC_1')\\
 &\quad +\lambda \varepsilon^2 Q \biggl( -\frac{1}{(1-\beta^2)^2} (2\beta S_V+
 (1+\beta^2) S_N)
 -\frac{\kappa(1+\beta\gamma)}{1-\beta^2}QS_1\biggr).
\end{align*}
Thus,
\[ 
\int_{\R} Q^{- \frac{1}{2}} F_W^2 \lesssim \varepsilon^2 \int_{\R} Q \left( (C_1')^2 + C_1^2 + \varepsilon^2 S_V^2 + \varepsilon^2 S_N^2 + \varepsilon^2 S_1^2 \right) .
\]
The integrals $\int_{\R} Q(C_1')^2$, $\int_{\R} QC_1^2$ and $\int_{\R} QS_1^2$ are estimated via \eqref{estim7SC} and \eqref{estim7SC'}. To control $\int_{\R} QS_V^2$ and $\int_{\R} QS_N^2$, combine \eqref{estim2DT} and \eqref{estim7SC}. Eventually,
\[
\int_{\R} Q^{- \frac{1}{2}} F_W^2 \lesssim \varepsilon^2 \int_{\R} Q^{\frac{1}{2}} ((W_2')^2 + (Z_2')^2+W_2^2 + Z_2^2) .
\]
Similarly,
\[
\int_{\R} y^2 F_W^2 \lesssim \varepsilon^2 \int_{\R} Q^{\frac{1}{2}}
((W_2')^2 + (Z_2')^2+W_2^2 + Z_2^2).
\]
Therefore, injecting these estimates in \eqref{viriel1},
\begin{align*}
\int_{\R} (W_2'')^2 + \int_{\R} (W_2')^2 
&\lesssim \varepsilon \left(\int_{\R} Q^{\frac{1}{2}} ((W_2')^2 + (Z_2')^2+W_2^2 + Z_2^2)\right)^\frac12 \\
&\qquad \times \left(\int_{\R} (W_2')^2 + \int_{\R} Q^{\frac{1}{2}} W_2^2
\right)^\frac12 \\
& \lesssim \varepsilon \int_{\R} \left((W_2')^2 +(Z_2')^2\right)
+ \varepsilon \int_{\R} Q^{\frac{1}{2}} (W_2^2 + Z_2^2) .
\end{align*}
By similar virial arguments on the relation \eqref{MZ}, we prove 
similarly that
\begin{equation}
\int_{\R} (Z_2'')^2 +\int_{\R} (Z_2')^2 \lesssim \varepsilon \int_{\R} \left((W_2')^2 +(Z_2')^2\right)
+ \varepsilon \int_{\R} Q^{\frac{1}{2}} (W_2^2 + Z_2^2) .
\end{equation}
Therefore, for $\omega >0$ small enough,
\begin{equation}
\int_{\R} \left((W_2'')^2 +(Z_2'')^2+ (W_2')^2+(Z_2')^2\right)
 \lesssim \varepsilon \int_{\R} Q^{\frac{1}{2}} (W_2^2 + Z_2^2) .
\label{viriel2}
\end{equation}
Now, we want to apply Lemma \ref{lemh} to the function $W_2$. 
We observe that
\[
 \int_{\R} hW_2 = \langle h , S^2 C_2 \rangle 
 = \langle (S^*)^2 h , C_2 \rangle = 
- \langle Q + 2yQ' , C_2 \rangle 
\]
and so by \eqref{estim3C2},
\[
\left| \int_{\R} hW_2 \right| \lesssim \varepsilon \left( \int_{\R} QC_1^2 \right)^\frac12 + \varepsilon^2 \left( \int_{\R} QC_2^2 \right)^\frac12.
\]
Lastly, by \eqref{estim7SC} and \eqref{estim8SC},
\[
\left| \int_{\R} hW_2 \right| \lesssim 
 \varepsilon \left( \int_{\R} Q^{\frac{1}{2}} (W_2^2 + Z_2^2) \right)^\frac12 .
\]
Therefore, it follows from Lemma \ref{lemh} applied to $W_2$ that
\begin{equation}\label{lemW2}
\int_{\R} Q^{\frac{1}{2}} W_2^2 \lesssim \left( \int_{\R} hW_2 \right)^2 + \int_{\R} (W_2')^2 \lesssim \varepsilon^2 \int_{\R} Q^{\frac{1}{2}} (W_2^2 + Z_2^2) + \int_{\R} (W_2')^2.
\end{equation}
Applying Lemma \ref{lemh} to the function $Z_2$, we prove similarly 
(using \eqref{estim3S2})
that
\begin{equation}
\int_{\R} Q^{\frac{1}{2}} Z_2^2 \lesssim 
\varepsilon^2 \int_{\R} Q^{\frac{1}{2}} (W_2^2 + Z_2^2) 
+\int_{\R} (Z_2')^2.
\label{lemZ2}
\end{equation}
Summing \eqref{lemW2} and \eqref{lemZ2},
and taking $\omega >0$ small enough, we eventually get that
\begin{equation}
\int_{\R} Q^{\frac{1}{2}} (W_2^2+Z_2^2)
 \lesssim \int_{\R} \left((W_2')^2+(Z_2')^2\right).
\label{lemWZ2}
\end{equation}
Combining \eqref{viriel2} and \eqref{lemWZ2}, we obtain
\[
 \int_{\R} \left((W_2'')^2 + (Z_2'')^2 +(W_2')^2 + (Z_2')^2 \right) 
 \lesssim \varepsilon \int_{\R} \left((W_2')^2 + (Z_2')^2 \right).
\]
For $\omega >0$ small enough,
we deduce that $W_2'=Z_2'=0$ and so $W_2=Z_2=0$.
From \eqref{estim7SC} and \eqref{estim8SC}, we get $C_1=S_1=C_2=S_2 = 0$. Moreover, from \eqref{Duha2}, we get $C_N=S_N=C_V=S_V=0$, which finishes the proof of the
theorem.

\subsection{Non-existence of resonance}\label{prk}
We justify Remark \ref{rk:2}.
Assume that 
\[
C_1,S_1,C_2 , S_2 , C_N , S_N , C_V , S_V\in L^\infty,\quad
C_1',S_1',C_2' , S_2' , C_N' , S_N' , C_V' , S_V'\in L^2
\]
satisfy the system \eqref{IM1}-\eqref{IM2} with $\lambda=1$.
From the system, we obtain directly that
$C_N , S_N , C_V , S_V\in L^2$.
All the arguments in \S\ref{S:3.2}-\ref{S:3.4} and \S\ref{S:3.6} can be reproduced in this setting. Moreover, since $W_2 = S^2C_2 = C_2'' - \frac{2Q'}{Q} C_2' + C_2$, we have $W_2 \in L^\infty$ with derivatives in $L^2$. Similarly, $Z_2 \in L^\infty$ with derivatives in $L^2$. See \cite[Proof of Corollary 2]{GR1} for a similar extension for NLS. The virial arguments of \S\ref{S:3.7} also apply to this more general case, and we find $C_1=S_1=C_2 =S_2=0$.

\section*{Appendix}

In this appendix, we derive rigorously Theorem \ref{thmPer} from Theorem \ref{thmIM}
by standard arguments.
We will denote by $\mathcal{D} ( \R )$ the set of smooth, compactly supported functions on $\R$.
Indexes will be used in order to highlight the appropriate variable.
For instance, we will write $\mathcal{D}_{s,y} = \mathcal{D} ( \R_s \times \R_y )$.

\begin{proof}
Take $(U_1,U_2,N,V) \in C^0 ( \R , H^1 ( \R )^2 \times L^2 ( \R )^2 )$ a time-periodic solution of the system \eqref{LinZ} and denote by $T>0$ its period. 
Take $n \geqslant 0$ and $A \gg 1$. We consider a sequence of smooth functions $\widetilde{\theta}_A \in \mathcal{D}_s$ such that 
\[ \widetilde{\theta}_A \underset{A \to + \infty}{\longrightarrow} \mathbbm{1}_{[0,T]} \text{ in $L_s^2$} \quad
\text{and} \quad 
\widetilde{\theta}_A' \underset{A \to + \infty}{\longrightarrow} \delta_0 - \delta_T \text{ in $\mathcal{D}_s'$.} \]
Set $\lambda_n = \frac{2 \pi n}{T}$ and $\theta_A (s) = \cos ( \lambda_n s ) \widetilde{\theta}_A (s)$. Take $\psi \in \mathcal{D}_y$. First, by \eqref{LinZ} we have
\begin{align}
\int_{\R_s \times \R_y} \theta_A '(s) \psi (y) U_1 (s,y) \, \text{d}s \, \text{d}y 
&= \left \langle \partial_s ( \theta_A (s) \psi (y)) , U_1 (s,y) \right \rangle_{\mathcal{D}_{s,y} , \mathcal{D}_{s,y}'} \nonumber \\
&= - \left \langle \theta_A (s) \psi (y) , \partial_s U_1(s,y) \right \rangle_{\mathcal{D}_{s,y} , \mathcal{D}_{s,y} '} \nonumber \\
&= - \left \langle \theta_A (s) \psi (y) , L_- U_2(s,y) \right \rangle_{\mathcal{D}_{s,y} , \mathcal{D}_{s,y}'} \nonumber \\
&= - \left \langle \theta_A (s) (L_- \psi )(y) , U_2 (s,y) \right \rangle_{\mathcal{D}_{s,y} , \mathcal{D}_{s,y}'} \nonumber \\
&= - \int_{\R_s \times \R_y} \theta_A (s) (L_- \psi )(y) U_2(s,y) \, \text{d}s \, \text{d}y.
\label{app1}
\end{align}
Second, by explicit differentiation and Fubini's theorem,
\begin{equation}\label{app2}
\begin{split}
& \int_{\R_s \times \R_y} \theta_A '(s) \psi (y) U_1(s,y) \, \text{d}s \, \text{d}y \\
&\quad = \int_{\R_s} \widetilde{\theta}_A '(s) \cos ( \lambda_n s) F_\psi (s) \, \text{d}s - \lambda_n \int_{\R} \widetilde{\theta}_A (s) \sin ( \lambda_n s) F_\psi (s) \, \text{d}s 
\end{split}
\end{equation}
where $F_\psi (s) := \int_{\R_y} \psi (y) U_1(s,y) \, \text{d}y$. Let us prove a useful regularity lemma before proceeding with the proof.

\begin{lemma} \label{lemapp}
The function $F_\psi$ is smooth on $\R$ and all of its derivatives (including $F_\psi$ itself) are $T$-periodic and bounded on $\R$.
\end{lemma}

\begin{proof}
In order to reproduce the proof for $N$ and $V$ instead of $U_1$, we only use the regularity $U_1 \in C^0 ( \R , L_y^2 )$. From $U_1 \in C^0 ( \R , L_y^2 )$ and the Cauchy-Schwarz inequality, it follows that $F_\psi \in C^0 ( \R )$. Besides, since $U_1$ is $T$-periodic in the variable $s$, the function $F_\psi$ is $T$-periodic and so bounded on $\R$. 
Now, in order to differentiate $F_\psi$, take any $\vartheta \in \mathcal{D}_s$. 
Reproducing the computation \eqref{app1} with the general function $\vartheta$ instead of $\theta_A$, we have
\begin{align*}
\left \langle  F_\psi' , \vartheta \right \rangle_{\mathcal{D}_s', \mathcal{D}_s} =& - \int_{\R_s} F_\psi (s) \vartheta '(s) \, \text{d}s = - \int_{\R_y \times \R_s} \vartheta '(s) \psi (y) U_1(s,y) \, \text{d}y \\
=& \int_{\R_s \times \R_y} \vartheta (s) (L_- \psi )(y) U_2(s,y) \, \text{d}s \, \text{d}y \\
=& \left \langle \vartheta (s) , \int_{\R_y} (L_- \psi ) (y) U_2(s,y) \, \text{d}y \right \rangle_{\mathcal{D}_s , \mathcal{D}_s'}.
\end{align*}
Hence $F_\psi '(s) = \int_{\R_y} (L_- \psi )(y) U_2 (s,y) \, \text{d}y$. 
As above, it follows that $F_\psi '$ is continuous, $T$-periodic and bounded on $\R$. 
We can iterate the computations above, using the system \eqref{LinZ}, and passing all the derivatives we need on $\psi$; we conclude that $F_\psi$ is smooth and that all of its derivatives are $T$-periodic, and thus bounded on $\R$.
\end{proof}
Now, we return to \eqref{app2}.
Since the function $\widetilde{\theta}_A'$ and its limit $\delta_0 - \delta_T$ as $A \to + \infty$ are compactly supported distributions, we can evaluate them against smooth functions (not necessarily compactly supported). Since $s \mapsto \cos ( \lambda_n s) F_\psi (s)$ is $T$-periodic, we have
\[
\int_{\R_s} \widetilde{\theta}_A '(s) \cos ( \lambda_n s) F_\psi (s) \, \text{d}s \underset{A \to + \infty}{\longrightarrow} - \left [ \cos ( \lambda_n s ) F_\psi (s) \right ]_0^T = 0.
\]
Moreover, 
\begin{equation*} 
\int_{\R} \widetilde{\theta}_A (s) \sin ( \lambda_n s) F_\psi (s) \, \text{d}s \underset{A \to + \infty}{\longrightarrow}  \int_0^T \sin ( \lambda_n s ) F_\psi (s) \, \text{d}s. 
\end{equation*}
Hence, letting $A \to + \infty$ in \eqref{app2} leads to
\begin{align}
\lim\limits_{A \to + \infty} \int_{\R_s \times \R_y} \theta_A '(s) \psi (y) U_1(s,y) \, \text{d}s \, \text{d}y 
&= - \lambda_n \int_0^T \sin ( \lambda_n s) F_\psi (s) \, \text{d}s \nonumber \\
&= - \lambda_n \int_0^T \int_{\R_y} \sin ( \lambda_n s) \psi (y) U_1(s,y) \, \text{d}s \, \text{d}y \nonumber \\
&= - \lambda_n \int_{\R_y} \psi (y) S_1^{(n)} (y) \, \text{d}y
\label{app3}
\end{align}
where $S_1^{(n)} (y) := \int_0^T \sin ( \lambda_n s ) U_1(s,y) \, \text{d}s$.

Now, we look at the limit of the right-hand term of \eqref{app1}. Since $\widetilde{\theta}_A$ and $\mathbbm{1}_{[0,T]}$ are compactly supported distributions and $s \mapsto \int_{\R_y} (L_- \psi )(y) U_2(s,y) \, \text{d}y = F_\psi '(s)$ is a smooth bounded function on $\R$,
we have
\begin{align} 
& \lim\limits_{A \to + \infty} \left ( - \int_{\R_s \times \R_y} \theta_A(s) (L_- \psi ) (y) U_2(s,y) \, \text{d}s \, \text{d}y \right ) \nonumber \\
& \quad= - \int_0^T \cos ( \lambda_n s) \int_{\R_y} (L_- \psi )(y) U_2(s,y) \, \text{d}s \, \text{d}y \nonumber \\
& \quad= - \int_{\R_y} (L_- \psi )(y) C_2^{(n)}(y) \, \text{d}y
\label{app4}
\end{align}
where $C_2^{(n)}(y) := \int_0^T \cos ( \lambda_n s ) U_2 (s,y) \, \text{d}s$.

Combining \eqref{app1}, \eqref{app3} and \eqref{app4} it follows that
\begin{equation*}
 \forall \psi \in \mathcal{D}_y , \quad \int_{\R_y} (L_- \psi )(y) C_2^{(n)} (y) \, \text{d}y = \lambda_n \int_{\R_y} \psi (y) S_1^{(n)} (y) \, \text{d}y
\end{equation*}
which means exactly that $L_- C_2^{(n)} = \lambda_n S_1^{(n)}$.

Setting 
\begin{align*}
& S_2^{(n)}(y) = \int_0^T \sin ( \lambda_n s ) U_2(s,y) \, \text{d}s, & C_1^{(n)}(y) = \int_0^T \cos ( \lambda_n s) U_1(s,y) \, \text{d}s, \\
& S_N^{(n)} (y)= \int_0^T \sin ( \lambda_n s) N(s,y) \, \text{d}s, &  C_N^{(n)}(y) = \int_0^T \cos ( \lambda_n s ) N(s,y) \, \text{d}s, \\
& S_V^{(n)}(y)= \int_0^T \sin ( \lambda_n s ) V(s,y) \, \text{d}s, & C_V^{(n)} (y) = \int_0^T \cos ( \lambda_n s ) V(s,y) \, \text{d}s,
\end{align*}
we prove similarly that $(\lambda_n , S_1^{(n)},C_1^{(n)},S_2^{(n)},C_2^{(n)},S_N^{(n)},C_N^{(n)},S_V^{(n)},C_V^{(n)})$ satisfy systems~\eqref{IM1} and \eqref{IM2}. 
(Note that $S_1^{(0)}=S_2^{(0)}=S_N^{(0)}=S_V^{(0)}=0$.)
Provided that $\omega >0$ is sufficiently small, it follows from Theorem \ref{thmIM} that $C_1^{(0)}=a_1Q'$, $C_2^{(0)}=a_2Q$, $C_N^{(0)}=0$, $C_V^{(0)}=0$ and for all $n \geqslant 1$,
 $C_1^{(n)}=S_1^{(n)}=C_2^{(n)}=S_2^{(n)}=C_N^{(n)}=S_N^{(n)}=C_V^{(n)}=S_V^{(n)}=0$. 

Now, we prove that $N=0$. Let $A<B$. We have
\[
\forall n \in \Z , \, \forall \psi \in C^0 ([A,B],\R) , \, \, \, \int_A^B \int_0^T N(s,y) e^{i \frac{2 \pi n}{T} s} \psi (y) \, \text{d}s \, \text{d}y = 0. 
\]
Since $N \in C^0 ( \R_s , L_y^2 )$, we have $N \in L^2 ([0,T] \times [A,B])$.
Since the family of functions $\left \{ e^{i \frac{2 \pi n} s} \psi (y) \, | \, n \in \Z \, \, \text{and} \, \, \psi \in C^0 ([A,B], \R ) \right \}$ is dense in $L^2 ([0,T] \times [A,B])$, it follows that $N=0$ in $L^2([0,T] \times [A,B])$. Since this result holds for any $A<B$ and that $N$ is $T$-periodic, we have $N(s)=0$ in $L_y^2$ for all $s \in \R$. Proceeding similarly with the functions $U_1-a_1 Q'$, $U_2-a_2Q$ and $V$, we obtain Theorem \ref{thmPer}.
\end{proof}

\section*{Acknowledgements}
The authors would like to thank the anonymous referees for their useful comments and corrections.

\end{document}